\documentclass[12pt,reqno]{amsart}
\usepackage{amsmath}
\usepackage{amsfonts}
\usepackage{amssymb}
\usepackage{amsthm}
\usepackage{graphicx}
\usepackage[utf8]{inputenc}
\usepackage{dsfont}
\usepackage{ae,aecompl}
\usepackage{pxfonts}
\usepackage{enumitem}
\usepackage[stretch=16,shrink=16,step=2,kerning=true,protrusion=true,final]{microtype} 
\usepackage{color}

\AtBeginDocument{%
   \def\MR#1{}
}

\usepackage[colorlinks=true, pdfstartview=FitV, linkcolor=darkblue, citecolor=darkred, urlcolor=cyan]{hyperref}
\definecolor{darkred}{RGB}{203,65,84}
\definecolor{darkblue}{RGB}{70,130,180}
\definecolor{darkgreen}{RGB}{2,100,64}

\newtheorem{theorem}{Theorem}[section]
\newtheorem{lemma}[theorem]{Lemma}
\newtheorem{proposition}[theorem]{Proposition}
\newtheorem{coro}[theorem]{Corollary}

\newtheorem{fact}[theorem]{Fact}
\newcounter{foo}
\newtheorem{theo}[foo]{Theorem}
\newtheorem{cor}[foo]{Corollary}

\theoremstyle{definition}
\newtheorem{definition}[theorem]{Definition}
\newtheorem{rema}[theorem]{Remark}
\newtheorem{remark}[theorem]{Remark}

\begin{document}
\newcommand{\interior}[1]{%
  {\kern0pt#1}^{\mathrm{o}}%
}
\newcommand{\ms}{\mathsf}
\newcommand{\inn}{_{n\in\mathbb N}}
\newcommand{\diam}{\operatorname{diam}}
\newcommand{\ad}{\operatorname{ad}}
\newcommand{\Ad}{\operatorname{Ad}}
\newcommand{\gp}{\mathbf{F}_\Theta}
\newcommand{\mT}{\mathcal T}
\newcommand{\mP}{\mathcal P}
\newcommand{\mc}{\mathcal}
\newcommand{\G}{\ms G}
\newcommand{\grf}{\pi_1(S)}
\newcommand{\bgrf}{\partial_\infty\pi_1(S)}
\newcommand{\psld}{\mathsf{PSL}_2(\mathbb R)}
\newcommand{\sld}{\mathsf{SL}_2(\mathbb R)}
\newcommand{\rp}{{\mathbf P}^1(\mathbb{R})}
\newcommand{\lix}[2]{\xi_\infty\left([#1 ,#2]\right)}
\newcommand{\tbw}{\begin{color}{darkred} To be written \end{color}}
\newcommand{\tbc}{\begin{color}{darkred} To be completed \end{color}}
\newcommand{\addw}{\begin{color}{darkblue} Add drawing \end{color}}
\newcommand{\mk}{\mathfrak}
\newcommand{\defeq}{\coloneqq}
\newcommand{\eqdef}{\eqqcolon}
\renewcommand{\leq}{\leqslant}
\renewcommand{\geq}{\geqslant}
\newcommand{\id}{\operatorname{id}}
\newcommand{\seq}[1]{
\{#1_m\}_{m\in\mathbb N}}
\newcommand{\seqm}[1]{
\{#1\}_{m\in\mathbb N}}
\newcommand{\mapping}[4]
{
\left\{
\begin{array}{rcl}
#1 &\rightarrow& #2\\
#3 &\mapsto& #4 
\end{array}
\right.
}
 
\title[Positive representations]{Positivity and representations  of surface groups}

\author[O. Guichard]{Olivier Guichard}
\address{Universit\'e de Strasbourg et CNRS\\
  IRMA, UMR 7501\\
  7 rue Descartes\\
F-67000 Strasbourg} \email{olivier.guichard@math.unistra.fr}

\author[F. Labourie]{Fran\c{c}ois Labourie}
\address{Universit\'e C\^ote d'Azur, CNRS, LJAD, 
    France } \email{francois.labourie@univ-cotedazur.fr}

\author[A. Wienhard]{Anna Wienhard}
\address{Max Planck Institute for Mathematics in the Sciences, Inselstr. 22, 04103 Leipzig
Germany}
\email{wienhard@mis.mpg.de}

\thanks{O.G.\  and F.L.\  acknowledge the funding of the  ANR grant DynGeo ANR-11-BS01-013 and of the Institut Universitaire de France. A.W.\ acknowledges funding by the DFG, Germany's Excellence Strategy EXC-2181/1 - 390900948 (the Heidelberg STRUCTURES Cluster of Excellence), by the ERC under
	ERC-Consolidator grant 614733 and ERC-Advanced Grant 101018839. She thanks the Klaus Tschira Foundation, the Hector Fellow Academy, and the Clay Foundation for support.  F.L.,  O.G., and A.W.\ acknowledge support from N.S.F  grants DMS 1107452, 1107263, 1107367 ``RNMS: GEometric structures And Representation varieties''. F.L.\ and A.W.\ acknowledge support from N.S.F  Grant No.~1440140 while in residence in Fall  2019, at the MSRI in Berkeley, F.L. acknowledges funding by the University of California at Berkeley and by the ERC under ERC-Advanced grant 101095722.}
\begin{abstract}
In~\cite{GuichardWienhard_pos, Guichard:2018vo} Guichard and Wienhard introduced the notion of $\Theta$-positivity, a generalization of Lusztig's total positivity to 
real Lie groups that are not necessarily split. 

Based on this notion,  we introduce in this paper  {\em $\Theta$-positive
  representations} of surface groups. We prove that $\Theta$-positive
representations of closed surface groups  are $\Theta$-Anosov. 
 This implies that $\Theta$-positive representations are discrete and faithful and that the set of
$\Theta$-positive representations is open in the representation variety. 
 
We further establish important properties on limits of $\Theta$-positive representations, proving that the set of $\Theta$-positive representations is closed in the set of representations containing a $\Theta$-proximal element. 
This is used in \cite{BGLPW_pos} to prove the closedness of the set of $\Theta$-positive representations. 
 
\end{abstract}

\maketitle
\section{Introduction} An important feature of Teichm\"uller space, seen as a connected component of the space of representations of the fundamental group of a closed connected orientable surface~$S$ of genus at least~$2$ in~$\psld$, is that it consists entirely of representations which are discrete and  faithful. These representations  are moreover quasi-isometries from~$\grf$ to~$\psld$.  This situation does not extend to the case of {\em any} semi-simple group, notably for simply connected complex ones, where the representation variety is irreducible as an algebraic variety~\cite{Rapinchuk:1996}. 

However, this phenomenon was shown to happen for {\em some} groups of higher
rank.
Two families of representation varieties of the fundamental group of~$S$
have been singled out as they contain connected components consisting entirely
of discrete and faithful representations:
\begin{itemize}
\item Hitchin components when $\G$ is a real split
  group~\cite{Fock:2006a} (the Hitchin components were defined in
  \cite{Hitchin:1992es}; the case of $\ms{SL}_3(\mathbb{R})$ was treated in
  \cite{Choi:1993vr} and the case of $\ms{SL}_n(\mathbb{R})$ in \cite{Labourie:2006}),	
\item spaces of maximal representations, which are defined when $\G$ is
  Hermitian~\cite{Burger:2010ty} (\cite{Goldman:1988vh} proves the
  case of $\ms{SL}_2( \mathbb{R})$, \cite{Burger:2005} studies maximal
  representations into symplectic groups).
\end{itemize}
When $\G$ is $\ms{PSL}_2(\mathbb R)$, the Hitchin component and the space of
maximal representations both agree with the Teichm\"uller space. 

The study of these two families is closely related to the theory of Anosov representations as
introduced in~\cite{Labourie:2006,Guichard:2012eg}. Being Anosov is a notion
defined for any reductive Lie group and with respect to a choice of a parabolic subgroup. Every Anosov representation is in particular faithful, discrete and a quasi-isometric embedding~\cite{Labourie:2005a,Guichard:2012eg,Delzant:2007wb}.

 Representations in the Hitchin components as well as maximal representations
 can be characterized in terms of equivariant curves from the boundary at
 infinity of $\pi_1(S)$ into an appropriate flag variety, which preserve some
 positivity.  
In~\cite{Labourie:2006} Labourie established the Anosov property for representations in Hitchin components for $\ms{SL}_n(\mathbb R)$ and showed that they furthermore admit hyperconvex boundary curves. In combination with  \cite{Guichard:2008tu} this gives a characterization of Hitchin components in terms of hyperconvex maps. 
The link to positivity was first made for split real groups in work of Fock and Goncharov~\cite{Fock:2006a}; in their groundbreaking work, they develop a new way to study moduli spaces of local systems, by looking at configuration spaces of (decorated) flags. They introduce new cluster coordinate systems, which are closely related to Lusztig's total positivity~\cite{Lusztig:1994}, and open new ways for quantizations. The positive points of their cluster varieties give rise to higher rank Teichm\"uller spaces, and, building on insight of~\cite{Labourie:2006}, lead to a characterization of Hitchin components in terms of positive boundary maps. 

The characterization of maximal representations in terms of positivity is given in work of Burger, Iozzi and Wienhard \cite{Burger:2010ty}, where the notion of positivity is based on the maximality of the Maslov index
  and related to Lie semigroups in~$\G$. 
 
In~\cite{Guichard:2018vo, GuichardWienhard_pos}, Guichard and Wienhard introduced the notion of {\em $\Theta$-positivity}. This notion extends Lusztig's total positivity to generalized flag manifolds associated with the parabolic defined by a set $\Theta$ of simple roots. 
They classified all possible simple Lie groups that admit a positive
structure relative to some~$\Theta$. These include real split Lie groups, for which $\Theta$-positivity
is Lusztig's total positivity, Hermitian Lie groups of tube type, where
$\Theta$-positivity is given by the maximality of the Maslov index, but also two
other families of Lie groups, namely the family of classical groups
$\mathsf{SO}(p,q)$ ---with $p\not=q$--- and an exceptional family consisting
of the real rank~$4$ form of $\mathsf F_4$,  $\mathsf E_6$,  $\mathsf E_7$,
and  $\mathsf E_8$ respectively. They conjectured that $\Theta$-positivity
provides the right underlying algebraic structure for the existence of
components made solely of discrete and faithful representations~\cite[Conjecture~19]{Wienhard_ICM}.  

A positive structure on $\G$ relative to~$\Theta$ implies in particular the existence of a
positive semigroup in the unipotent radical of the parabolic group~$\ms
P_\Theta$, which then leads to the notions of {\em positive triples} and {\em
  positive quadruples} (as well as positive tuples) in the flag variety $\gp\simeq \G/\ms P_\Theta$. In the basic example of $\G=\psld$ and $\gp=\rp$, a triple is positive if it consists of pairwise distinct points and a quadruple is positive if it is cyclically ordered. 

Let us give a geometric picture of positivity in the flag variety~$\gp$.
For this let~$a$ and~$b$ be two points in~$\gp$ which are transverse to each other. 
Then $\Theta$-positivity provides the existence of preferred connected
components of the set of all points in~$\gp$ that are transverse to both~$a$
and~$b$. These preferred components are called  {\em diamonds} (with extremities~$a$ and~$b$). They are several, at least two,  disjoint diamonds with given extremities. The semigroup property alluded to before translates into a nesting property of diamonds:  if $c$~is a point in a diamond $V(a,b)$ with extremities~$a$ and~$b$, then there is exactly one diamond $V(c,b)$ (with extremities~$c$ and~$b$) included in $V(a,b)$. These nesting properties of diamonds play an important role in our arguments.  

 If $a$ and $b$ are transverse, and $c$ belongs to a diamond with extremities~$a$ and~$b$, we say the triple $(a,b,c)$ is {\em positive}. Similarly, one can define positive quadruples using configurations of diamonds (see Figure~\ref{fig:pos-quadr} and Definition~\ref{def:pos}). We show in Section~\ref{sec:property} that being positive is invariant under all permutations for a triple, and  invariant under the dihedral group for a quadruple.

We define a map $\xi$ from a cyclically ordered set $A$ to $\gp$ to be {\em positive} if $\xi$ maps triples of pairwise distinct points to positive triples and cyclically ordered quadruples to positive quadruples. 

This allows us to define the notion of a {\em $\Theta$-positive representation}: 
A representation $\rho\colon\pi_1(S) \to \G$ is {\em $\Theta$-positive} if there exist a non-empty subset $A$ of $\bgrf$, invariant by $\pi_1(S)$, and  a $\rho$-equivariant positive boundary map from $A$ to $\gp$.

We prove  
\begin{theo}\label{thm_intro:Anosov}
Let  $\G$ be a semi-simple Lie group that admits a positive
structure relative to~$\Theta$. Let $\rho$ be a $\Theta$-positive representation from $\pi_1(S)$ to~$\G$.

Then $\rho$ is a $\Theta$-Anosov representation.  	
\end{theo}

As a direct consequence we obtain that a $\Theta$-positive representation is faithful with discrete image, its orbit map into the symmetric space is a quasi-isometric embedding and the boundary map extends uniquely to a H\"older map~\cite{Labourie:2005a,Guichard:2012eg,Delzant:2007wb,Bridgeman:2015ba}.

Theorem~\ref{thm_intro:Anosov} provides a general proof of the Anosov property
for all Hitchin representations and all maximal representations.
This is especially relevant for the case  of the Hitchin component of
$\mathsf{SO}(p,p)$ and of $\mathsf{F}_4$, $\mathsf{E}_6$, $\mathsf{E}_7$, and
$\mathsf{E}_8$ and the case of maximal representations into the exceptional Hermitian Lie group of tube type, which
cannot be tightly embedded into $\mathsf{Sp}_{2n}(\mathbb R)$
\cite{Burger_tight, Hamlet_1, Hamlet_2}.  The Anosov property was established for all maximal
representations which tightly embed into $\mathsf{Sp}_{2n}(\mathbb R)$
in~\cite{Burger:2005} and for the Hitchin component of \(\mathsf{SL}_n(\mathbb{R})\) in
\cite{Labourie:2006}, from which follows the Anosov property for the Hitchin
components of \(\mathsf{Sp}_{2n}(\mathbb{R})\), \(\mathsf{SO}(p,p+1)\), and
\(G_2\). Fock and Goncharov established a related key property: for  every
Hitchin representation, there exists  a continuous, transverse (and positive)
boundary map \cite[Theorem~7.2]{Fock:2006a}; 
from this, the Anosov property
can be established for Zariski dense Hitchin representations  using for example \cite[Theorem~4.11]{Guichard:2012eg}. 

Using their work on amalgamation of Anosov representations, Dey and Kapovich \cite[Section 6]{Dey-Kapovich} established also the Anosov property for all Hitchin components for all real split groups.

Using the openness of the set of $\Theta$-Anosov representations, a further consequence of Theorem~\ref{thm_intro:Anosov} is the following 
\begin{cor}\label{cor_intro:open}
The set of $\Theta$-positive representations $\operatorname{Hom}_{\Theta\textrm{-pos}}(\grf,\G)$ is an open subset in the set of all homomorphisms $\operatorname{Hom}(\grf,\G)$.
\end{cor}

To show that the set of $\Theta$-positive representations indeed give rise to higher Teichm\"uller spaces it  remains to prove that the set of $\Theta$-positive representations is closed. We establish essential steps in this direction. For this we consider 
 the set $\operatorname{Hom}^\Theta(\pi_1(S),\G)$ of homomorphisms~$\rho$
 of~$\pi_1(S)$ in~$\G$  such that the image of $\rho$ contains a
 $\Theta$-loxodromic element (i.e.\ an element having both attracting and
 repelling fixed points in the flag variety $\gp$ associated to
 $\Theta$). Observe that Proposition \ref{pro:BL} clarifies the relation with
 the Zariski closure, and in particular the set
 $\operatorname{Hom}^\Theta(\pi_1(S),\G)$ contains the representations with Zariski dense images. 
We establish in Proposition~\ref{prop:nonpara} that  $\operatorname{Hom}_{\Theta\textrm{-pos}}(\Gamma,\G)$ is a subset of $\operatorname{Hom}^\Theta(\pi_1(S),\G)$. We show

\begin{theo}\label{thm_intro:nonpara}
The set of $\Theta$-positive representations 
$\operatorname{Hom}_{\Theta\textrm{-pos}}(\grf,\G)$  is a nonempty union of connected components of 
$\operatorname{Hom}^\Theta(\grf,\G)$.
\end{theo}

In the case when $\G$ is locally isomorphic to $\mathsf{SO}(p,q)$, $p\leq q$,
Beyrer and Pozzetti~\cite{BeyrerPozzetti} recently proved the {\em closedness}
of the set of $\Theta$-positive Anosov representations in
$\operatorname{Hom}(\pi_1(S),\G)$, thus by Theorem~\ref{thm_intro:Anosov} also
the closedness of the set of $\Theta$-positive representations. They derive
this as a consequence of a family of collar lemmas and fine properties of the
boundary maps they establish. 
  In~\cite{BGLPW_pos}, together with Beyrer and Pozzetti, we prove collar lemmas in full generality for $\Theta$-positivity, which in combination with Theorem~\ref{thm_intro:nonpara} establishes that the set of $\Theta$-positive representations is closed in $\operatorname{Hom}(\pi_1(S),\G)$. Thus with Corollary~\ref{cor_intro:open} and Theorem~\ref{thm_intro:Anosov} $\Theta$-positive representations give rise to connected components consisting entirely of discrete and faithful representations.

Note that special $\Theta$-positive representations arise from positive
embeddings of $\mathsf{SL}_2(\mathbb{R})$ into $\G$.  These positive
embeddings of $\mathsf{SL}_2(\mathbb{R})$ can be produced 
explicitly using specific ``positive'' nilpotent element in
the Lie algebra of~$\G$. They have the property that the embedding induces a
positive map from~$\rp$ into~$\gp$. We call the image of such a map a {\em
  positive circle}. Every group $\G$ admitting a positive
structure relative to~$\Theta$ contains a special (conjugacy class of) $\Theta$-principal  $\mathsf{SL}_2(\mathbb{R})$. The circles associated to this $\Theta$-principal $\mathsf{SL}_2(\mathbb{R})$ play an important role in some of our arguments. 
Precomposing a positive embedding $\mathsf{SL}_2(\mathbb{R})$ into  $\G$ with a discrete embedding of $\pi_1(S)$ into $\mathsf{SL}_2(\mathbb{R})$, we obtain a $\Theta$-positive representation. 

Recently, Bradlow, Collier, Garc\'ia-Prada, Gothen, and
Oliveira~\cite{BCGGO_general} developed the theory of magical
$\mathfrak{sl}_2$-triples, which is very closely related to the theory of
$\Theta$-positivity. In fact a real simple Lie group is associated to a
magical $\mathfrak{sl}_2$-triple if and only if it admits a positive
structure relative to~$\Theta$. Using methods from the theory of Higgs bundles,
  they parametrize special connected components $\mathcal{P}_e(S,\G)$, called Cayley components. 
We expect these connected components $\mathcal{P}_e(S,\G)$ to consist entirely of $\Theta$-positive representations, and furthermore to coincide with the set of $\Theta$-positive representations. 
We discuss the relation between $\mathcal{P}_e(S,\G)$ and $\Theta$-positive representations in Section~\ref{sec:connected}.

\smallskip 

\noindent
{\bf Acknowledgements:} We thank Michel Brion, Steve Bradlow, Brian Collier,
Beatrice Pozzetti, and J\'er\'emy Toulisse for interesting discussions regarding
the topics of this paper. We also thank Nicolas Tholozan and Tengren Zhang for pointing out 
mistakes in previous versions.

We thank the referees who pointed out some inaccuracies in the previous version.
\vskip 0.2 truecm
\noindent
{\bf Outline of the paper:} In Section~\ref{sec:def}, we recall the necessary algebraic material from \cite{Guichard:2018vo,GuichardWienhard_pos} and introduce the main definitions: diamonds, positive configurations, positive circles and positive maps. In Section~\ref{sec:property}, we prove three propositions concerning combinatorial properties of configurations, proper inclusion of diamonds and extension of positive maps ---some of the properties proved here are also in \cite{GuichardWienhard_pos}, but the proofs in the present paper are geometric, while those in \cite{GuichardWienhard_pos} are algebraic. In Section~\ref{sec:metric}, we introduce the diamond metric on diamonds and establish its properties. With these preparations we prove Theorem~\ref{thm_intro:Anosov} and Corollary~\ref{cor_intro:open} in Section~\ref{sec:posano}, Theorem~\ref{thm_intro:nonpara} in Section~\ref{sec:closedness}. In Section~\ref{sec:connected} we discuss the connection with the Cayley components introduced in ~\cite{BCGGO_general}.

\tableofcontents

\section{Definitions}\label{sec:def}
\subsection{Lie algebra notations}

Let $\G$ be a semi-simple group.

The \emph{roots} of~$\G$ are the nonzero weights under the adjoint action of a
Cartan subspace~$\mathfrak{a}$ on the Lie algebra~$\mathfrak{g}$ of~$\G$. They
form a root system~$\Sigma \subset \mathfrak{a}^*$ (nonreduced in some cases)
and the choice of a linear ordering on~$\mathfrak{a}^*$ gives rise to
the set~$\Sigma^+$ of positive roots, and to the set~$\Delta$ of simple
roots. The $\alpha$-weight space will be denoted by~$\mathfrak{g}_\alpha
\subset \mathfrak{g}$.

The parabolic subgroups of~$\G$ are the subgroups conjugated to one of the
standard parabolic subgroups~$\mathsf{P}_\Theta$ (for~$\Theta$ varying in the
subsets of~$\Delta$); namely $\mathsf{P}_\Theta$~is the normalizer in~$\G$ of
the Lie algebra
$\mathfrak{u}_\Theta \defeq \bigoplus_{\alpha\in \Sigma^+ \smallsetminus
  \operatorname{span}(\Delta\smallsetminus\Theta)} \mathfrak{g}_\alpha$.  The
unipotent radical of~$\mathsf{P}_\Theta$ is the subgroup
$\mathsf{U}_\Theta = \exp( \mathfrak{u}_\Theta)$.  A parabolic subgroup is its
own normalizer so that the space~$\gp$ of parabolic subgroups conjugated
to~$\mathsf{P}_\Theta$ is isomorphic to $\G/\mathsf{P}_\Theta$.

The space~$\gp$ is also naturally $\G$-isomorphic to the $\G$-orbit (for the
adjoint action)
of~$\mathfrak{u}_\Theta$ in the space~$\mathbf{L}$ of Lie subalgebras of~$\mk g$. The group
$\operatorname{Aut}(\mk g)$ of automorphisms of~$\mk g$ also acts on~$\mathbf{L}$ and the actions of~$\G$ and of
$\operatorname{Aut}(\mk g)$ on this space are related via the adjoint action
seen as an homomorphism $\G \to \operatorname{Aut}(\mk g)$. For~$\psi$ in $\operatorname{Aut}(\mk g)$ and for $\mathfrak{u}$
in the $\G$-orbit of~$\mathfrak{u}_\Theta$ ({\it i.e.}\ $\mathfrak{u}$~belongs
to~$\gp$), the algebra $\psi(\mathfrak{u})$ may not belong to~$\gp$; in fact
$\psi(u)$~belongs to ${\bf F}_{\psi_*(\Theta)}$ where $\psi_*\colon \Delta\to
\Delta$ denotes the action of~$\psi$ on the set of simple roots (or on the
Dynkin diagram). There is thus a subgroup $\operatorname{Aut}_0(\mk g)$
of $\operatorname{Aut}(\mk g)$ that acts (transitively) on~$\gp$. This group
$\operatorname{Aut}_0(\mk g)$ has better transitive properties than~$\G$, \emph{e.g.}\
it will act transitively on the diamonds that are introduced later. We 
will therefore use several times $\operatorname{Aut}_0(\mk g)$ instead of~$G$.

Two parabolic subgroups~$\mathsf{P}$ and~$\mathsf{P}'$ are called
\emph{transverse} or \emph{opposite} if their intersection $\mathsf{P} \cap
\mathsf{P}'$ is a reductive subgroup ({\it i.e.}\ the unipotent radical of
this intersection is trivial);
this is equivalent to having $\operatorname{UniRad}( \mathsf{P}) \cap
\mathsf{P}'=\{1\}$. In that case, there exists $\Theta\subset \Delta$ such that
the pair $( \mathsf{P}, \mathsf{P}')$ is conjugated to $( \mathsf{P}_{\Theta},
\mathsf{P}^{\mathrm{opp}}_{\Theta})$ where
$\mathsf{P}_{\Theta}^{\mathrm{opp}}$ is the normalizer of $\bigoplus_{\alpha\in \Sigma^+
  \smallsetminus \operatorname{span}(\Delta\smallsetminus\Theta)} \mathfrak{g}_{-\alpha}$. The
intersection~$\mathsf{L_\Theta} \defeq \mathsf{P}_{\Theta} \cap
\mathsf{P}^{\mathrm{opp}}_{\Theta}$ is a Levi factor of~$\mathsf{P}_{\Theta}$
(and of~$\mathsf{P}_{\Theta}^{\mathrm{opp}}$).

We will always work with a parabolic subgroup $\mathsf{P}\simeq
\mathsf{P}_\Theta$ such that $\mathsf{P}_\Theta$ is conjugated to its
opposite~$\mathsf{P}_{\Theta}^{\mathrm{opp}}$; in this situation it makes sense to look at transverse elements in
$\gp  \simeq \G/\mathsf{P}_\Theta$. In particular we will use the
following notation, for $x$ in $\gp$,
\begin{align*}
\mathsf{P}_x&{\defeq}\operatorname{Stab}(x)\ ,\\
{\ms U}_x&{\defeq}\operatorname{UniRad}(\ms P_x)\ ,\\
\Omega_x&{\defeq}\{y\in {\gp }\mid y\text{ is transverse to } x\}\ ,\\
S_x&{\defeq}\gp\smallsetminus\Omega_x\ .
\end{align*}
We will sometimes use that, if $a$ and $b$ are transverse points,
  then $\mathsf{L}_{a,b}\defeq \mathsf{P}_a \cap \mathsf{P}_b$ is a Levi factor of~$\mathsf{P}_a$ and~$\mathsf{P}_b$.
Recall that $\Omega_x$ is an open orbit of $\ms U_x$ and that $S_x$ is a proper algebraic subvariety of~$\gp$.

Given a point~$a$ in~$\gp$, a {\em unipotent pinning}, or \emph{U-pinning} at~$a$,  is an identification~$s$ of~$\ms U_\Theta$ with~$\ms U_a$ that exponentiates
an isomorphism from~$\mathfrak{u}_\Theta$ to~$\mathfrak{u}_a$ which itself is induced
by the
restriction of an
automorphism of the Lie algebra~$\mathfrak{g}$ ({\it i.e.}\ an element of $\operatorname{Aut}_0(\mk g)$). Observe that there are
finitely many U-pinnings up to the action of $\ms L_\Theta$. 

\subsection{Cones and semigroup}

\begin{definition}\label{defi:semigroup}\cite[Theorem~12.2]{GuichardWienhard_pos}
  A \emph{positive structure} with respect to $\gp$ (or a
  \emph{positive structure relative to~$\Theta$}) is a 
  semigroup~$\mathsf{N}$ of~$\mathsf{U}_{\Theta}$  such that,
  denoting~$x$ and~$y$ the points of~$\gp$ whose stabilizers are~$\mathsf{P}_\Theta$
  and~$\mathsf{P}_{\Theta}^{\mathrm{opp}}$ respectively, $\mathsf{N}\cdot y$ is a
  connected 
  component of $\Omega_x \cap \Omega_y$.
   \end{definition}

   In this case, $\mathsf{N}$~is invariant by conjugation by the connected
   component~$\mathsf{L}^{\circ}_{\Theta}$
   of~$\mathsf{L}_\Theta$ and is a sharp semigroup: for any
   $h$, $k$ in $\overline{\ms N}$, if $hk=1$, then $h=k=1$ ({\it i.e.}\ the only
   invertible element in~$\overline{\mathsf{N}}$ is the identity element).

We shall see that given $a$ and $b$ transverse to each other in $\gp$ 
and an identification of $\ms U_\Theta$ with $\ms U_a$ (\emph{i.e.}\ a U-pinning) which sends $\ms N$ to  a subgroup $\ms N_a$ of $\ms U_a$, then $\ms N_a\cdot b$ is a connected component of $\Omega_a\cap\Omega_b$. 

In~\cite{GuichardWienhard_pos} it is proved that, up to the action of $\operatorname{Aut}_0(\mk g)$, the semigroup~$\mathsf{N}$ in the definition is unique. 

We first present some conclusions of the construction of the semigroup~$\mathsf{N}$ that we are going to use in this paper, then concentrate on the notions of {\em diamonds} and {\em positive configurations} that play a crucial role in this paper.

\subsubsection{The parametrization of the positive semigroup}\label{sec:param} 
Theorem~4.5 of~\cite{Guichard:2018vo} and Theorem~1.3 of~\cite{GuichardWienhard_pos}
give a precise description of the possible parametrizations of the
semigroup~$\mathsf{N}$. We recall here the material necessary for our purpose.

\begin{fact}
  \label{fact:param-posit-semi}
  There exist $N\geq 1$ and $\mathsf{C}$ a $\mathsf{L}^{\circ}_{\Theta}$-invariant cone in $(
  \mathfrak{u}_\Theta)^N$ such that the map
  \begin{align*}
    (\mathfrak{u}_\Theta)^N & \longrightarrow \mathsf{U}\\
    (x_1, \dots, x_N) & \longmapsto \exp(x_1)\cdots \exp(x_N)
\intertext{  induces by restriction a $\mathsf{L}^{\circ}_{\Theta}$-equivariant diffeomorphism}
\Psi\colon \mathsf{C}& \longrightarrow \mathsf{N}\ .
  \end{align*}

  Furthermore the stabilizer in~$\mathsf{L}^{\circ}_{\Theta}$ of any point~$h$
  in~$\mathsf{C}$, and therefore of any point~$n$ in~$\mathsf{N}$, is a
  compact subgroup of~$\mathsf{L}^{\circ}_{\Theta}$.
\end{fact}
The closure $\overline{\mathsf{C}}$ is also $\mathsf{L}^{\circ}_{\Theta}$-invariant and 
Definition~\ref{defi:semigroup} implies that the cone~$\overline{\mathsf{C}}$
is salient, {\it i.e.}\ the intersection of~$\overline{\mathsf{C}}$
and~$-\overline{\mathsf{C}}$ is reduced to~$\{0\}$.

\begin{remark}\label{rem:detail_on_cones}
  More precisely, for every~$\alpha$ in~$\Theta$ an
  $\mathsf{L}^{\circ}_{\Theta}$-invariant cone~$C_\alpha$ has been chosen in the $\mathsf{L}$-irreducible
  factor of~$\mathfrak{u}_\Theta$ corresponding to~$\alpha$ (and $C_\alpha$~is
  open in that factor) and 
  we have that  $\mathsf{C}=C_1 \times C_2 \times \cdots \times C_N$ where $N$~is the length of the longest element in a finite Coxeter group
  associated with~$\Theta$ and, for every $i=1,\dots, N$, $C_i$ is
  one of the cones~$C_\alpha$ \cite[Theorem~1.3]{GuichardWienhard_pos}.
\end{remark}

\subsection{Diamonds}
Let~$a$ and~$b$ be two transverse points in~$\gp$.
\begin{definition}
  A {\em diamond} with {\em extremities~$a$ and~$b$}, associated with a U-pinning~$s_a$ at~$a$, is the subset
\begin{equation*}
s_a(\ms N)\cdot b\ .
\end{equation*}
\end{definition}
The terminology {\em diamond} was coined in~\cite{Labourie:2020tv} in the context of $\G=\ms{SO}(2,n)$. To give an idea, in that context $\gp$ is covered by charts which are identified with the Minkowski space $\mathbb R^{1,n-1}$. Then a diamond is, in a suitable chart, the intersection of the future time cone~$F^+$ of~$a$, with the past time cone~$F^-$ of~$b$.

  In that case there are precisely two diamonds with given extremities. More generally, from \cite[Corollary 13.5]{GuichardWienhard_pos}, it follows that the number of diamonds with given extremities is $2^{\sharp \Theta}$.

\begin{rema}
  We observe that diamonds are semi-algebraic sets and make sense over a real
  closed field.
\end{rema}
We list some first properties of diamonds that are direct consequences of the
definition or are proved
in~\cite[Section~13]{GuichardWienhard_pos}.

\begin{proposition}\label{pro:diam}
\begin{enumerate}
\item A diamond with extremities~$a$ and~$b$ is a connected component
  of $\Omega_a\cap\Omega_b$.
\item Given a diamond $s_a(\ms  N)\cdot b$, there exists a U-pinning~$s_b$
  at~$b$ such that
  \begin{equation*}
    s_a(\ms N)\cdot b=s_b(\ms N)\cdot a\ .
  \end{equation*}
   
\item Given any diamond  $V(a,b)=s_a( \mathsf{N})\cdot b$ then $a$ belongs to
  the closure of $V(a,b)$.
   
\end{enumerate}
\end{proposition}
\begin{proof}
The first item is a consequence of~\cite[Theorems 1.3 and 1.4]{GuichardWienhard_pos}.  The
second item is a consequence of~\cite[Proposition
13.1]{GuichardWienhard_pos}.

The third item follows from the fact that the identity belongs to the closure
of $\ms N$. 
 
\end{proof}

We also remark that
 
\begin{proposition}\label{prop:opp-diamonds}
  Given a diamond~$V$ there is a unique diamond $V^*$ satisfying the following
  property: given
  any U-pinning~$s_b$ at~$b$, if $V=s_b(\ms N)\cdot a$ then
  $V^*=s_b(\ms N^{-1})\cdot a$. The diamond~$V^*$ is called the {\em opposite
    diamond} to~$V$ (one says also that the diamond~$V^*$ is opposite to~$V$). A diamond and its opposite are disjoint, more precisely
  any point in $V$ is transverse to any point in~$V^*$.
\end{proposition} 
\begin{proof}
We just have to remark that the definition of the opposite diamond does not
depend on the choices. More precisely, given two U-pinnings~$s_a$ and~$s_b$, if
\[V=s_b(\ms N)\cdot a=s_a(\ms N)\cdot b \ ,
\]
then 
\[s_b(\ms N^{-1})\cdot a=s_a(\ms N^{-1})\cdot b\ ;
\]
this holds by~\cite[Section~13]{GuichardWienhard_pos}.

The last point comes from~\cite[Remark 4.9]{Guichard:2018vo} and from \cite[Section 13.6]{GuichardWienhard_pos}. In particular, if $x\in V$, then $x=s_b(n)\cdot a$ with $n\in \ms N$, while if  $y\in V^*$, then $y=s_b(m^{-1})\cdot a$ with $m\in \ms N$. Thus
\begin{equation*}
 x=s_b(nm)\cdot y\ .
\end{equation*}
Since $\mathsf{N}$~is a semigroup, this means that $x$~belongs to a diamond
with extremities~$y$ and~$b$. By the first point of Proposition~\ref{pro:diam}, $x$~is transverse to~$y$.
\end{proof}

As a consequence of the proposition, if $c$~is an element in a diamond with extremities~$a$ and~$b$, we will denote by 
\begin{itemize}
\item $V_c(a,b)$ the unique diamond containing~$c$ with extremities~$a$ and~$b$.
\end{itemize}
 
Note that, for any~$d$ in~$V_c(a,b)$, one has $V_d(a,b)=V_c(a,b)$; also
$V_c(b,a) = V_c(a,b)$.
 
In addition, $V^*_c(a,b)$ is the diamond opposite to the diamond containing~$c$.

As an immediate consequence of the semigroup property we obtain the following result that we shall use freely:
\begin{lemma}[\sc Nesting property]\label{lem:semigroup}
  Let $c$ be a point in a diamond with extremities~$a$ and~$b$. 
 \begin{enumerate}
 	\item\label{item1:lem:semigroup} Then there exists a
  unique  diamond $V(a,c)$ with extremities~$a$ and~$c$ such that
  \begin{equation*}
    V(a,c)\subset V_c(a,b) \ .
  \end{equation*}
   Furthermore there is a neighborhood $U$ of $a$ in~$\gp$ such that $$
  U\cap V(a,c)=U\cap V_c(a,b). $$
\item\label{item2:lem:semigroup}  Moreover, if $V(c,b)$ is the unique diamond with extremities~$c$ and~$b$ included in $V_c(a,b)$ then 
  \begin{equation*}
  	V(a,c)\cap V(c,b)= \emptyset\ .
  \end{equation*}
  \item\label{item3:lem:semigroup}  Finally $a$~belongs to the opposite diamond $V^*(c,b)$ and the
    diamond $V(a,c)$ is contained in~$V^*(c,b)$.
   
 \end{enumerate} 
\end{lemma}

\begin{figure}[ht] 
  \begin{center}
    \includegraphics[width=3 in]{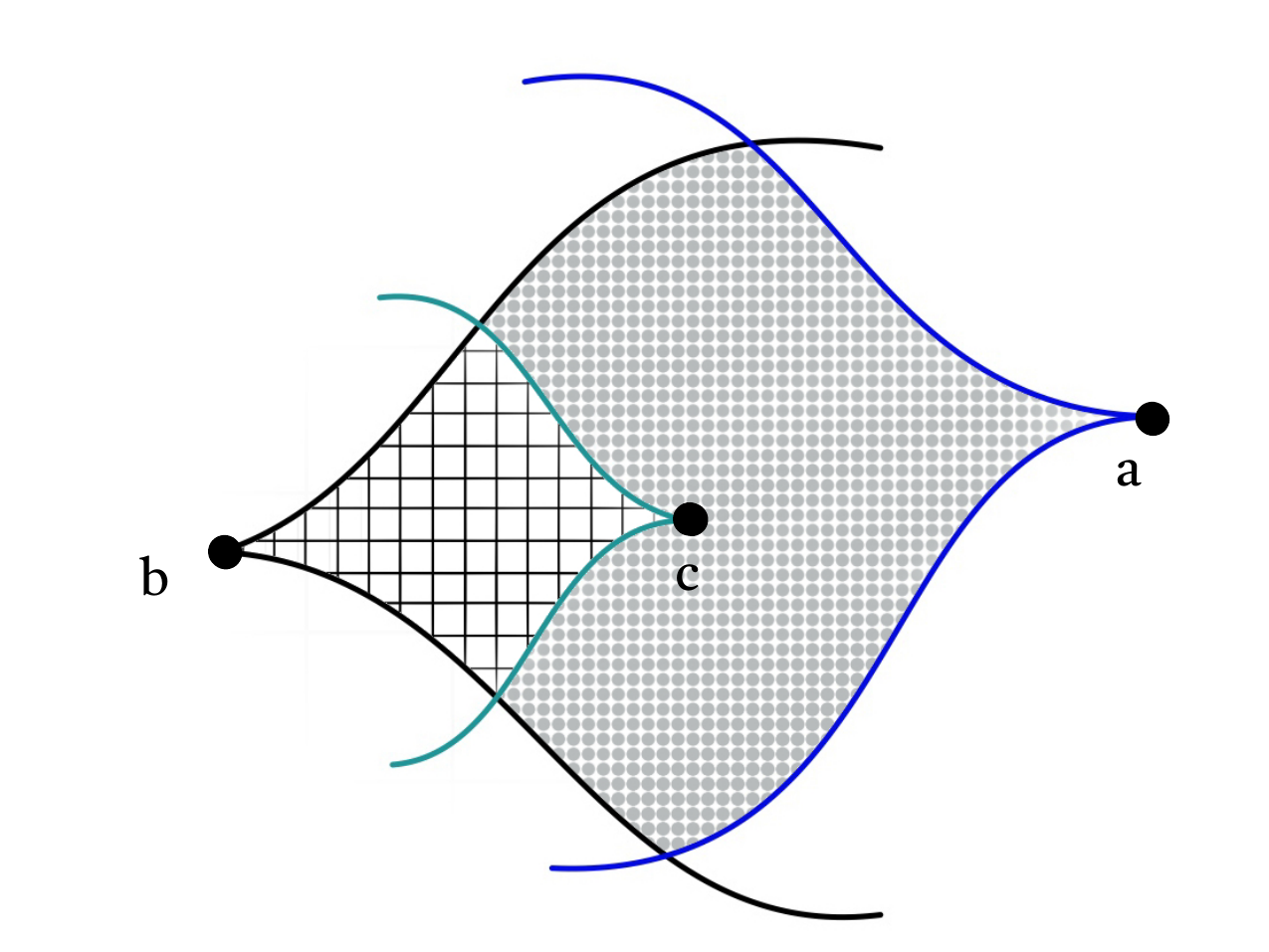}
     \caption{The nesting of $V(c,b)$ in $V_c(a,b)$}\label{fig:nesting}
  \end{center}
  \end{figure}

The proof of Lemma~\ref{lem:semigroup} will use the following statement.

\begin{lemma}
  \label{lem:connected-comp-local0}
  Let~$a$ and~$b$ be two transverse points of~$\gp$ and $V_0$ a diamond with extremities $a$ and $b$. Then, there
  exists a basis ~$\mathcal{B}$  of
  neighborhoods of~$a$ such that for every~$U$ in~$\mathcal{B}$, the intersection $U\cap V_0$ is connected and nonempty.
 \end{lemma}

\begin{proof}
  Up to acting by an element of~$\G$, we can assume that the stabilizer of~$b$
  is~$\mathsf{P}_\Theta$ and that the stabilizer of~$a$
  is~$\mathsf{P}_{\Theta}^{\mathrm{opp}}$. The map from $\mathfrak{u}_\Theta$ to $\Omega_b$ given by  $
  x \mapsto \exp(x)\cdot a$ is a
  $\mathsf{L}_\Theta$-equivariant diffeomorphism.

  Consider the decomposition $\mathfrak{u}_\Theta= \bigoplus_{i} V_i$ into 
  $\mathsf{L}_\Theta$-irreducible factors. Let us fix an auxiliary Euclidean
  norm $\Vert\cdot\Vert$ on~$\mathfrak{u}_\Theta$ such that the previous
  decomposition is orthogonal. There is a one-parameter subgroup $\Lambda=\{
  \lambda_t\}_{t\in \mathbb{R}}$ of~$\mathsf{L}_\Theta$ such that, for all~$i$ and for
  all~$v$ in~$V_i$, $\lambda_t\cdot v = e^{n_i t}v$ for some positive numbers $n_i$.

  Let $S$ be the unit sphere in $\mathfrak{u}_\Theta$ for
  $\Vert\cdot\Vert$. Then the map from $S\times \mathbb{R}$ to $
  \Omega_b\smallsetminus \{a\}$, given by  
  $$ 
  g\colon   (v,t)\mapsto
  \exp(\lambda_t \cdot v) \cdot a \ , $$ 
  is a diffeomorphism  satisfying that for all $v$ in $S$, all real numbers  $t$ and  $s$  
  $$g(v,t+s)=\lambda_s \cdot g(v,t)\ .$$
  Thus since $\Omega_a\cap\Omega_b$ is 
  $\Lambda$-invariant we have the following property:  for
  all~$v$ in~$S$ and all~$t$, $t'$ in~$\mathbb{R}$, $g(v,t)$ belongs
  to~$\Omega_a\cap\Omega_b$ if and only if  $g(v,t')$ belongs
  to~$\Omega_a\cap\Omega_b$.

 Thus, there is a connected open $\Omega_0$ in $S$, such that the diamond $V_0$ ---being a connected component of $\Omega_a\cap \Omega_b$ by the first item of Proposition~\ref{pro:diam}--- is the image of $\Omega_0\times \mathbb{R}$ by the map $g$.  Let finally  $O_t$ be the images 
  of $S\times (-\infty, t)$ by $g$ and $U_t=O_t\cup\{a\}$. Then $\{U_t\}_{t\in\mathbb R}$ is a family of neighborhoods of $a$ with the wanted property.
\end{proof}

\begin{proof}[Proof of Lemma~\ref{lem:semigroup}]   Let us first construct diamonds $V^0(c,b)$ and $V^0(a,c)$ included in $V_c(a,b)$.   Let us write $V_c(a,b)=\ms N_b\cdot a= \ms N_a\cdot b$ and  
 consider the diamonds 
 \begin{equation*}
   V^0(c,b)=\ms N_b\cdot c\ , \ V^0(a,c)= \ms N_a\cdot c\ .
 \end{equation*}
By construction  $c=n_b\cdot a= n_a\cdot b$ with $n_b\in\ms N_b$ and $n_a\in\ms N_a$. By the semigroup property 
\begin{equation*}
\ms N_b\cdot n_b\subset \ms N_b\ , \ \ms N_a\cdot n_a\subset \ms N_a\ ,
\end{equation*}
which leads to the inclusions
\begin{equation*}
V^0(c,b)\subset V_c(a,b) \ , \ \ V^0(a,c)\subset V_c(a,b) \ .
\end{equation*}

We now prove that these specific diamonds are disjoint.
By the construction  and the inclusion above both $V^0(a,c)$ and $V^0(b,c)$ are
connected components of $V_c(a,b)\smallsetminus S_c$. It follows that they
are either equal or disjoint. By the sharpness property of~$\mathsf{N}$, the identity element
does not belong to the closure of $\ms N\cdot n_a$. Let thus $O$ be an open
set in $\ms U_a$ containing the identity and with trivial intersection with
$\ms N\cdot n_a$. Then $O\cdot b$ is a neighborhood of~$b$ that does not
intersect $\ms N_a\cdot c=\ms N_an_a\cdot b$. Thus $b$~does not belong to
the closure of $V^0(a,c)$. From the last item of Proposition~\ref{pro:diam},
$V^0(a,c)$ is hence different from $V^0(c,b)$ and by the above discussion they are disjoint:
\[V^0(a,c)\cap V^0(b,c)=\emptyset\ . 
\]
This concludes item~(\ref{item2:lem:semigroup}) of the lemma.

Let us prove next the  existence of the neighborhood $U$. Denote for any open set~$V$, $\partial V\defeq \overline{V}\smallsetminus V$ and denote~$V^c$ the complementary of~$V$ and recall that
\begin{align*}
\partial(V\cap W)\subset &\overline{V}\cap \overline{W}\smallsetminus (V\cap W)\\
=& \overline{V}\cap \overline{W} \cap (W^c \cup V^c)\\
=& (\overline{V}\cap \overline{W} \cap W^c) \cup (\overline{V}\cap \overline{W} \cap  V^c)\\
=&(\partial V\cap \overline{W})\cup (\partial W\cap \overline{V})\ .
\end{align*}
Let $V(a,c)$ be any diamond with extremities~$a$ and~$c$ included in
$V_c(a,b)$.  Let~$U$ be a neighborhood of~$a$ such that
\begin{itemize}
\item the intersection of~$\overline{U}$ with $S_c\cup S_b$ is empty,
\item $V_c(a,b)\cap U$ is connected and nonempty.
\end{itemize}
The existence of this open set~$U$ is guaranteed by Lemma~\ref{lem:connected-comp-local0}.
From the first item we have that
\begin{align*}
\partial (V(a,c)\cap U)\subset((\partial V(a,c))\cap \overline{ U}) \cup (\overline{V}(a,c)\cap\partial U) \ \subset& \ (S_a\cup\partial U)\ . 
\end{align*}
From the inclusion $V(a,c)\subset V_c(a,b)$ we have 
\begin{align*}
\partial (V(a,c)\cap U)\subset& \ \overline{V_c(a,b)\cap U}\ .
\end{align*}
Since $S_a\cup\partial U$ is included in the complementary of $V_c(a,b)\cap U$ we furthermore have 
\begin{align*}
(S_a\cup\partial U)\cap (\overline{V_c(a,b)\cap U})\ \subset\  & \partial (V_c(a,b)\cap U)\ .
\end{align*}
Thus combining these inclusions, we get 
\begin{align*}
\partial (V(a,c)\cap U)\subset (S_a\cup\partial U) \cap (\overline{V_c(a,b)\cap U})	\subset \partial (V_c(a,b)\cap U) \ .	
\end{align*}
Now a simple connectedness argument show that if $A$ and $B$ are two open sets, with $B$ connected, $A\subset B$ and $\partial A\subset \partial B$, then $A=B$. Thus, in our case, 
$$V(a,c)\cap U=V_c(a,b)\cap U\not=\emptyset .$$
Since this is true for all diamonds with extremities $a$ and $c$ included in
$V_c(a,b)$ and since diamonds with the same extremities  are either disjoint
of equal, we finally conclude that there is a unique diamond with extremities
$a$ and $c$ included in $V_c(a,b)$. This concludes
item~(\ref{item1:lem:semigroup}) of the lemma.

For the last item, observe that $$a= n_b^{-1}c\in
\mathsf{N}_b^{-1}c=V^*(c,b)\ .$$
Since $a$~belongs to the closure of $V(a,c)$, we have hence $V(a,c)\cap
V^*(c,b)\neq \emptyset$. Furthermore $V(a,c)\subset \Omega_c$ and
$V(a,c)\subset V(a,b)\subset \Omega_b$; this means that the connected set
$V(a,c)$ is contained in $\Omega_a\cap \Omega_b$. Therefore $V^*(b,c)$ is the
connected component of  $\Omega_a\cap \Omega_b$ containing $V(a,b)$: this is
the sought for inclusion.
\end{proof}

\subsection{Positive configurations}
The following definition plays a central role in this article:

Let $p\geq
3$ and equip $\{1,\ldots, p\}$ with the usual cyclic order. 
\begin{definition}[\sc Positive configuration]
  \label{def:pos}
We say that a configuration $(a_1,\ldots, a_p)$ in $\mathbf{F}^{p}_{\Theta}$ is {\em positive}, if  there
exist diamonds $V_{i,j}$ with extremities~$a_i$ and~$a_j$ for all $i\not=j$ such that 
\begin{enumerate}
\item\label{item:1:def:pos} $V_{i,j}=V^*_{j,i}$,
\item\label{item:2:def:pos} $a_j$ belongs to $V_{i,k}$, if $(i,j,k)$ is cyclically oriented,
 
\item we have $V_{i,j}\subset V_{i,k}$ and $V_{j,k}\subset V_{i,k}$, if
  $(i,j,k)$ is cyclically oriented.
 
\end{enumerate}
\end{definition}

\begin{figure}[ht]
  \includegraphics[width=5 in]{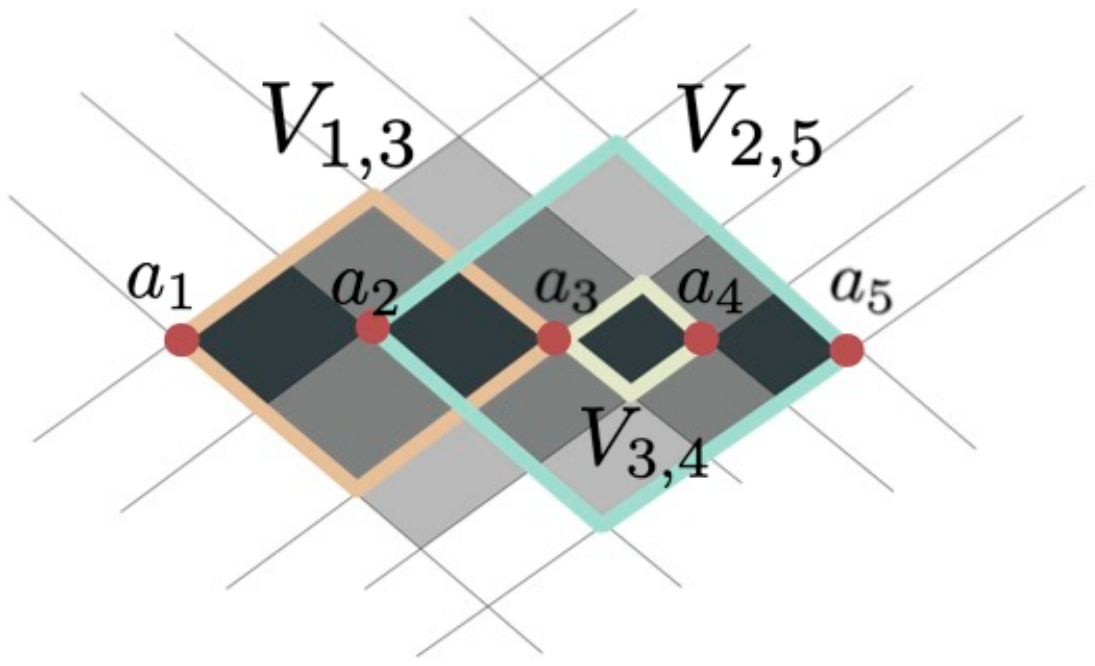}
  \caption{A positive $5$-configuration and some diamonds}
  \label{fig:pos-quadr}
\end{figure}

Proposition~\ref{pro:basic} will give easier criteria to understand positive
triples and quadruples and will show that the definition is equivalent to the definition given in the introduction.

Observe that, by properties~(\ref{item:2:def:pos}) and~(\ref{item:1:def:pos})
above, the choice of $V_{i,k}$  among diamonds with extremities~$a_i$ and~$a_k$ is forced by the cyclic ordering. Furthermore, the fact that
$V_{i,k}$ does not depend on the index~$j$ between~$i$ and~$k$ involves the
positivity of a subquadruple.
It thus follows that if $(a_1,\ldots,a_p)$ is such that every cyclically oriented subquadruple is positive then  $(a_1,\ldots,a_p)$ is positive.

By construction, every subconfiguration of a positive configuration is
positive. On the real projective line, a configuration of $p$ points with
$p>3$ is positive exactly if it is cyclically oriented, and a triple is
positive if it consists of pairwise distinct elements. 

Moreover 
\begin{proposition}
  Positivity of configurations is invariant under cyclic permutation and under
  the
  order reversing permutation. In particular
  \begin{enumerate}
  \item to be positive for a triple is invariant under all permutations,
  \item to be positive for a quadruple is invariant under the dihedral group.
  \end{enumerate}
\end{proposition} 
\begin{proof}
  The definition is invariant under cyclic transformations. If $\sigma_0$ is  the reverse ordering,  we choose the new diamonds $V^\circ_{i,j}=V^*_{\sigma_0(i),\sigma_0(j)}$. 
\end{proof}

\subsection{Positive circles and  $\mathsf{PSL}_2({\mathbb R})$}

Let $\ms H$ be a subgroup in  $\G$ locally isomorphic to $\psld$. An {\em $\ms
  H$-circle} in $\gp$ is a closed  $\ms H$-orbit, it can be parametrized by a
{\em circle map} which is a $\ms H$-equivariant map from ${\rp}$ to $\gp$. The
group $\ms H$ is {\em proximal} if it contains a proximal element in~$\gp$,
\emph{i.e.}\ an element having an attracting fixed point on~$\gp$.

\begin{proposition}[\sc $\ms H$-circle]\label{pro:pref-circle}
  Given a positive structure, there exists $\mathcal H$, 
  an $\operatorname{Aut}_0(\mk g)$-orbit of pairs $(\ms H,C)$ such that $\ms
  H$ is a
  subgroup  of $\G$  locally isomorphic to $\psld$, $C$ is an $\ms H$-circle, satisfying the following properties
  \begin{enumerate}
  \item\label{item1:pro:pref-circle} $\ms H$  has a compact centralizer
    in~$\G$;
  \item \label{item:pro:pref-circle:Ctransverse} Any two distinct points
    on~\(C\) are transverse;
  \item\label{item3:pro:pref-circle} Given 
    a diamond~$V$ with
    extremities~$a$ and~$b$, there exists $(\ms H,C)$
    in $\mathcal H$ with~$C$ containing~$a$ and~$b$, and such that
    $C$~intersects the diamond~$V$. Furthermore
    \begin{itemize}[leftmargin=*]  
    \item If $c$ is a point in $C$ different from~$a$ and~$b$, then
      $(a,c,b)$ is a positive triple and
      $$V_c (a,b)\cap C\  \hbox{ and } V^*_c (a,b)\cap C ,$$
      are the two connected components of $C\smallsetminus\{a,b\}$.
    \item If  $d$~belongs to the connected component of
      $C\smallsetminus\{c,b\}$ not containing~$a$, then
      $$ V_d(b,c)\subset V_d(a,b)\ .$$
    \end{itemize}
  \item\label{item4:pro:pref-circle}  Given any three pairwise distinct points $a$, $b$, and~$c$
    in~$\gp$. Then there is at most one element $(\mathsf{H},C)$
    of~$\mathcal{H}$ such that $C$~contains $a$, $b$, and~$c$.
  \end{enumerate}     
\end{proposition}
Note that in point~(\ref{item3:pro:pref-circle}) $V$~needs to be equal to~$V_c(a,b)$ or~$V_{c}^{*}(a,b)$.

\begin{proof}
  Let $s_b$ be a U-pinning at~$b$ such that $V= s_b( \mathsf{N})\cdot
  a$.
  One just picks the Lie subgroup associated with an
  $\mathfrak{sl}_2$-triple given by the Jacobson--Morozov theorem applied to
  a nilpotent element~$x$ chosen as in \cite[Section 7]{GuichardWienhard_pos} so that $\ms N_2=s_b(\exp(\mathbb R_{>0}
  x))$ is included in~$\ms N$. The corresponding $\mathfrak{sl}_2$-triple is the $\Theta$-principal $\mathfrak{sl}_2$-triple introduced in 
 \cite[Section 7]{GuichardWienhard_pos}. Item~(\ref{item1:pro:pref-circle})
 follows now from Fact~\ref{fact:param-posit-semi}. 

 Item~(\ref{item:pro:pref-circle:Ctransverse}) follows from \cite[Lemma 7.2]{GuichardWienhard_pos}.

   For item~(\ref{item3:pro:pref-circle}), the existence is immediate by
 $\operatorname{Aut}_0(\mk g)$-transitivity. The two connected components of
 $C\smallsetminus\{a,b\}$ are $\ms N_2\cdot a$ and ${\ms N}_{2}^{-1}\cdot a$
 and are thus included in diamonds opposite to each other.
  
  Moreover, for the last statement in item~\ref{item3:pro:pref-circle}, let us write $c=n\cdot a$ with~$n$ in~$\ms N_2$.  Observe that, by a deformation argument 
  $$
  V_{n\cdot d}(c,b)=V_d(c,b)\ \text{and}\ V_d(a,b)=V_c(a,b)\ .
  $$ 
  Then $V_d(a,b)=\mathsf{N}\cdot a$ and
  $$
  V_d(c,b)=V_{n\cdot d}(c,b)= n V_d(a,b)= n\ms N\cdot a\subset \mathsf{N}\cdot a=V_d(a,b)\ ,
  $$
  where the inclusion holds by the semigroup property.

  For the item~(\ref{item4:pro:pref-circle}), let us consider~$Z$ the stabilizer of~$a$, $b$, and~$c$. Then
  $Z$ is precisely the stabilizer in $\mathsf{L}_{a,b}$ of the element~$n$
  in~$\mathsf{N}_2$ such that $c=n\cdot a$. Since $\mathsf H$ is determined by $\mathsf N_2$, this implies that $Z$ is in fact
  the centralizer of~$\mathsf{H}$.
  This concludes the
  proof.
   
\end{proof}

\begin{remark}
  More detail on the construction of the $\Theta$-principal  $\mathfrak{sl}_2$-triple can be found in
  \cite[Section 7]{GuichardWienhard_pos}. 
  Note that there are others $\mathfrak{sl}_2$-triples which induce positive
  maps from ${\rp}$ to $\gp$. For example, if $\G$ is a split real Lie group,
  the principal  $\mathfrak{sl}_2$-triple gives rise to such a   map. 
\end{remark}

We fix once and for all such an $\operatorname{Aut}_0(\mk g)$-orbit $\mathcal
H$. 

As an important example of positive configuration, we have
\begin{proposition}\label{pro:conf-circ} Let $(\mathsf{H}, C)$ be in~$\mathcal{H}$. Any cyclically ordered configuration
  of points on~$C$ is positive.
\end{proposition}
\begin{proof} It is enough to prove the results for triples and quadruples. 

Let first $(a_0,a_1,a_2)$ be a triple on~$C$.  By
item~(\ref{item3:pro:pref-circle}) in Proposition~\ref{pro:pref-circle}, $a_{i+1}$ belongs to a diamond with
extremities~$a_i$ and~$a_{i+2}$. Let us define (where indices are taken modulo~$3$)
$$
V_{i,i+2}\defeq V_{a_{i+1}}(a_i,a_{i+2})\ , V_{i,i+1}\defeq V^*_{i+1,i}\ .
$$

Then the properties of Definition~\ref{def:pos} are obviously
satisfied, and the triple is positive. 

Let now consider  $(a_0,a_1,a_2,a_3)$  a  quadruple on~$C$,  
such  that $a_{i+1}$ and $a_{i+3}$ belongs to different components of
$C\smallsetminus\{a_{i},a_{i+2}\}$.   Observe that by a deformation
argument we have 
$$
V_{a_{i+2}}
(a_i,a_{i+3})
=V_{a_{i+1}}(a_i,a_{i+3})\ .
$$
We now  define 
\begin{align*}
	V_{i,i+2}&\defeq V_{a_{i+1}}(a_i,a_{i+2})\ , \\
V_{i,i+3}&\defeq V_{a_{i+2}}(a_i,a_{i+3})
=V_{a_{i+1}}(a_i,a_{i+3})\ , \\
V_{i,i+1}&\defeq V^{*}_{i+1,i}\ .
\end{align*}
 
It then follows from  item~(\ref{item3:pro:pref-circle}) of Proposition~\ref{pro:pref-circle} that $V_{a_{i+3}}(a_i,a_{i+2})=V^*_{a_{i+1}}(a_i,a_{i+2})$, and thus that 
 $V_{i+2,i} = V_{i,i+2}^{*}$. 
 
 From the equality $V_{i,i+1}=V_{i+1,i}^*$ and from  item~(\ref{item3:pro:pref-circle})  of
 Proposition~\ref{pro:pref-circle},  $V_{i+1,i+2}\cap C$ is the connected component of
 $C\smallsetminus\{a_{i+1},a_{i+2}\}$ not containing $a_{i}$ and
 $a_{i+3}$. Let~$d$ be in $V_{i+1,i+2}\cap C$, then  
 $$
 V_{i+1,i+2}=V_d(a_{i+1},a_{i+2})\subset V_d(a_{i},a_{i+3})=V_{i,i+3}\ ,
 $$ 
 where, for the inclusion, we applied twice the last part of the item~(\ref{item3:pro:pref-circle}) of Proposition~\ref{pro:pref-circle}.
 
This concludes the proof.
\end{proof}

\subsection{Positive  maps}
        
Let $S$ be a cyclically ordered set containing at least three points. 
\begin{definition}[\sc Positive map]
A map~$f$ from~$S$ to~$\gp$ is {\em positive} if the image of every cyclically ordered  quadruple is a positive quadruple, and the image of every cyclically ordered triple is a positive triple.\footnote{When $S$~has more than three points, the second requirement follows from the first.}
\end{definition}

Observe then  that the image of every cyclically ordered configuration by a
positive map is a positive configuration. 

By Proposition~\ref{pro:conf-circ}, for any $(\mathsf{H}, C)$
in~$\mathcal{H}$, $C$ ---seen as a map from~$\rp$ to~$\gp$---
 is positive.

\section{Properties of positivity}\label{sec:property} We prove in this
section, three main propositions concerning positivity:
\begin{itemize}[leftmargin=*]
\item The first one, Proposition~\ref{pro:basic}, gives various combinatorial
  properties of positive triples, quadruples and configurations;
	\item The second one, Proposition~\ref{pro:bounded}, gives information about the limit of diamonds included in a given diamond;
	\item The last one, Proposition~\ref{pro:ext}, shows that positive maps share the property of monotone maps: they coincide on a dense subset with a left-continuous positive map.
\end{itemize}
We also establish that certain elements in~$\G$ are $\gp$-proximal using positivity.

Several of the combinatorial properties of positive configurations have been
addressed in \cite{GuichardWienhard_pos} with a more algebraic approach, for reader's convenience, we
provide here geometric proofs using the nesting properties of diamonds.

\subsection{Combinatorics of positivity}

The  next proposition gives fundamental properties of
positive triples and quadruples.
\begin{itemize}
 \item 	The first one gives an easy criterion for positivity of triples,  while the second and third concern quadruples. In particular, this shows that the definition of positivity given in the introduction is equivalent to Definition~\ref{def:pos}.
 \item   The fourth one gives a recursive way to build positive configuration.
 \item The fifth and sixth give ``exclusion'' properties that are important in
   the study of positivity though they are not used in this paper.
 \end{itemize}

We are going to prove this proposition and its corollary  in the context of a
group defined over $\mathbb R$, although by Tarski--Seidenberg
Theorem, the statements
will be true over every real closed field.

\begin{proposition}[\sc Combinatorial properties]\label{pro:basic}
  \begin{enumerate}[leftmargin=*]
  \item\label{it:basic1} Assume that~$a$ and~$b$ are transverse and that
    $c$~belongs to a diamond with extremities~$a$ and~$b$, then $(a,b,c)$ is
    positive.
  \item\label{it:basic3}  Assuming $(a,x_0,b)$ and $(a,y_0,b)$ are positive
    then $(a,x_0,b,y_0)$  is positive if and only if
    $V_{x_0}(a,b)=V^*_{y_0}(a,b)$.
  \item\label{it:basic2}  Assuming $(a,c,b)$ is positive and $d$ belongs to
    $V^*_a(c,b)$, then $(a,c,d,b)$ is positive.
  \item\label{it:basic4} More generally, assume that $(x_0,x_1,\ldots,x_p)$ is a positive
    configuration and that $y\in V^*_{x_2}(x_0,x_1)$ then
    \[	(x_0,y,x_1,\ldots,x_p)\ ,
    \]
    is a positive configuration.
  \item\label{it:basic5}    If $(a,b,c,d)$
    is
    positive, then $(a,c,b,d)$ is not positive. \label{axpos:ex3}
  \item \label{it:basic6} Let $x_0$, $x_1$, and
  $x_2$ be three points such that $(a,x_i,b)$ is positive ($i=0,1,2$), then the three  quadruples $(a,x_0,b,x_1)$, $(a,x_1,b,x_2)$, and
  $(a,x_2,b,x_0)$ cannot all be positive. 
\end{enumerate}
\end{proposition}

Finally we have,

\begin{coro}[\sc Necklace property] \label{coro:necklace}
  Let $(a,b,c)$ be a positive triple. Let $\alpha$, $\beta$, and $\gamma$ be
  elements of $V_a (b,c)$, $V_b(a,c)$, and $V_c(a,b)$ respectively. Then the triple
  $(\alpha,\beta,\gamma)$ is positive.
\end{coro}

The proof of this proposition and of Corollary~\ref{coro:necklace} 
will be given in Section~\ref{sec:basic}.
It is important to remark that all these properties are true for
configurations in~$\rp$.  As an immediate consequence of
Proposition~\ref{pro:basic}, we also have that the intersections of diamonds arising from positive configurations of points are
diamonds:
\begin{coro}[\sc Intersection of diamonds]
  \label{coro:inter-diam}
  Let \((a,b,c,d)\) be a positive quadruple.  Then the following equality
  holds
  \[ V^{*}_{a}(b,c) = V^{*}_{d}(a,c) \cap V^{*}_{a}(b,d)\ .\]
\end{coro}
\begin{proof}
  The inclusion \(V^{*}_{a}(b,c) \subset V^{*}_{d}(a,c) \cap V^{*}_{a}(b,d)\)
  is a direct consequence of the nesting properties of diamonds
  (point~(\ref{item3:lem:semigroup}) of Lemma~\ref{lem:semigroup}).

  Let us prove the reverse inclusion.  Let~\(x\) be in \(V^{*}_{d}(a,c) \cap
  V^{*}_{a}(b,d)\).  This means that the quadruples \((a,b,x,d)\) and
  \((a,x,c,d)\) are positive.  From the first one we get that \(b\)~belongs to
  \(V_{d}^{*}(a,x)\) and from the second one we get that \(V_{d}^{*}(a,x) =
  V_{c}^{*}(a,x)\).  Point~(\ref{it:basic4}) of Proposition~\ref{pro:basic}
  applied with \(y=b\) and \((x_0,x_1, x_2)=(a,x,c)\) gives that \((a,b,x,c)\)
  is positive and hence that \(x\)~belongs to \(V^{*}_{a}(b,c)\).
\end{proof}

\subsubsection{Triples and quadruples}\label{sec:triples-quadruples}
Before addressing the proof of Proposition~\ref{pro:basic}, we establish a
number of preliminary statements. 
\begin{lemma}\label{lem:triple}
  A triple $(a_0,a_1,a_2)$ is positive if and only if $a_0$, $a_1$, $a_2$ 
  belong to diamonds with extremities $(a_1,a_2)$, $(a_2,a_0)$ and
  $(a_0,a_1)$ respectively.
\end{lemma}
\begin{proof} We just need to prove the ``if'' part. Let, for $i=0,1,2$
  (indices are taken modulo~$3$)
  \begin{equation*}
V_{i,i+1}\defeq V^*_{a_{i+2}}(a_i,a_{i+1})\ , \ V_{i,i+2}\defeq V_{a_{i+1}}(a_{i},a_{i+2}). 
\end{equation*}
Observe that 
\begin{equation*}
  V_{i,i+1}  =V^*_{i+1,i}\ .
\end{equation*}
Then Lemma~\ref{lem:semigroup}.(\ref{item3:lem:semigroup}) provides all the necessary inclusions needed to
prove that the triple is positive.
\end{proof}

The following lemma gives a way to go from positive triples to positive quadruples.

\begin{lemma}\label{lem:subtrip}
  Let $(a_0,a_1,a_2,a_3)$ be a quadruple. Assume that all subtriples are
  positive.  Then the quadruple  $(a_0,a_1,a_2,a_3)$ is positive, if and only
  if, for all~$i$ (indices are taken modulo~$4$)
  \begin{align}
    a_i&\in V^*_{a_{i+2}}(a_{i+1},a_{i+3})\ , \label{eq:quad00} \\
    a_{i+2}&\in V_{a_{i+1}}(a_{i},a_{i+3})\ . \label{eq:quad000} 
  \end{align}
\end{lemma}
\begin{proof} 
The ``only if'' part follows from the definition. It remains to prove the
``if'' part.
Let 
\begin{align*}
  V_{i,i+1}&\defeq V^*_{a_{i+2}}(a_i,a_{i+1}) = V^*_{a_{i+3}}(a_i,a_{i+1})\ ,\\
  V_{i,i+2}&\defeq V_{a_{i+1}}(a_{i},a_{i+2})=
             V^*_{a_{i+3}}(a_{i},a_{i+2})\ ,\\
  V_{i,i+3}&\defeq V_{a_{i+1}}(a_i,a_{i+3})=
             V_{a_{i+2}}(a_{i},a_{i+3})\ ,		
\end{align*}
where in the second line 
we used the hypothesis~\eqref{eq:quad00}, while in the first and last lines 
we used the hypothesis~\eqref{eq:quad000} and the fact that if  $d$ belongs to $V_a(b,c)$ then $V_d(b,c)=V_a(b,c)$. Hence by definition
\begin{equation*}
V_{i,i+1}=V^*_{i+1,i}\ , V_{i,i+2}=V^*_{i+2,i}\ . \label{eq:subtrip0}	
\end{equation*}
It thus follows that  for all $i$ and $j$,
\begin{equation}
V_{i,j}=V^*_{j,i}\ .\label{eq:subtrip2}	
\end{equation}

From the positivity of the subtriple $(a_i,a_{i+1},a_{i+2})$, we get the inclusions
\begin{equation}
V_{i,i+1}\subset  V_{i,i+2}\ , \ \ V_{i+1,i+2}\subset  V_{i,i+2}\ .\label{eq:subtrip1}	
\end{equation}

From the positivity of the triple $(a_{i},a_{i+1},a_{i+3})$ we get the inclusions
\begin{align}
V_{i,i+1}=V^*_{a_{i+3}}(a_i,a_{i+1})&\subset V_{a_{i+1}}(a_i,a_{i+3})=V_{i,i+3}\ ,\label{eq:subtrip3}	\\
V_{i+1,i+3}=V^*_{a_i}(a_{i+1},a_{i+3})&\subset V_{a_{i+1}}(a_{i},a_{i+3})=V_{i,i+3}\ .\label{eq:subtrip4}
\end{align}
Similarly the positivity of the triple $(a_{i},a_{i+2},a_{i+3})$ yields
\begin{align}
V_{i+2,i+3}=V^*_{a_{i}}(a_{i+2},a_{i+3})&\subset V_{a_{i+2}}(a_i,a_{i+3})=V_{i,i+3}\ ,\label{eq:subtrip5}	\\
V_{i,i+2}=V^*_{a_{i+3}}(a_{i},a_{i+2})&\subset	V_{a_{i+2}}(a_{i},a_{i+3})=V_{i,i+3}\ .\label{eq:subtrip6}
\end{align}
All together the equation~\eqref{eq:subtrip2} as well as the inclusions	\eqref{eq:subtrip1}, \eqref{eq:subtrip3}, \eqref{eq:subtrip4}, \eqref{eq:subtrip5}, and~\eqref{eq:subtrip6} prove that $(a_0,a_1,a_2,a_3)$ is   a positive quadruple. \end{proof}

\subsubsection{Deformation lemmas}

We need to prove some deformation lemmas.

\begin{lemma}[\sc Deforming triples]\label{lem:open-trip}
  Let $a(t)$, $b(t)$, and $c(t)$ be continuous arcs from $[0,1]$ to $\gp$ such
  that
  \begin{enumerate}
  \item for all $t$ in $[0,1]$, 
    $a(t)$, $b(t)$, and $c(t)$
    are pairwise transverse, 
  \item the triple  $(a(0),b(0),c(0))$ is  positive.
  \end{enumerate}
  Then, for all $t$,  $(a(t),b(t),c(t))$ is a positive triple.
\end{lemma}

\begin{proof}
  The hypothesis tells us that
  \begin{equation*}
    c(t)\in\Omega_{a(t)}\cap\Omega_{b(t)}\ , \ \ 
    a(t)\in\Omega_{c(t)}\cap\Omega_{b(t)}\ , \ \ 
    b(t)\in\Omega_{a(t)}\cap\Omega_{c(t)}\ .
  \end{equation*}
  By hypothesis, 
  there are diamonds $V(a(0),b(0))$, $V(c(0),b(0))$, and $V(c(0),a(0))$ such
  that
  \begin{equation*}
    c(0) \in V(a(0),b(0))\ ,\ \ 
    a(0) \in V(c(0),b(0))\ ,\ \ 
    b(0) \in V(c(0),a(0))\ .
  \end{equation*}
We can extend these to continuous maps $t\mapsto V(a(t),b(t))$,
$t\mapsto V(c(t),b(t))$, and $t\mapsto V(c(t),a(t))$ in the space of diamonds (one always has
that $V(e,d)$ is a diamond with extremities~$e$ and~$d$).
We now use the fact that a diamond with extremities~$e$ and~$d$  is a
connected component of $\Omega_e\cap\Omega_d$
(Proposition~\ref{pro:diam}). Then  by continuity, 
for all $t$
\begin{equation*}
	c(t)\in V(a(t),b(t))\ ,\ \ 
	a(t)	\in V(c(t),b(t))\ ,\ \ 
	b(t)	\in V(c(t),a(t))\ .
\end{equation*}
Thus the result follows from Lemma~\ref{lem:triple}. \end{proof}

Similarly
\begin{lemma}[\sc Deforming quadruples]\label{lem:open-quad}
  Let $\gamma$ and $\eta$ be continuous arcs from $[0,1]$ to~$\gp$ such that
  there exist~$a$ and~$b$ in~$\gp$ satisfying
  \begin{enumerate}
  \item for all $t$ in $[0,1]$, 
    $a,\gamma(t),b,
      \eta(t)$ are pairwise transverse,
  \item the quadruple $(a,\gamma(0),b,\eta(0))$ is positive.
  \end{enumerate}
  Then, for all $t$, $(a,\gamma(t),b,\eta(t))$ is a positive quadruple.
\end{lemma}
\begin{proof}
  By applying Lemma~\ref{lem:open-trip}, we obtain that all the subtriples of $(a,\gamma(t),b, \eta(t))$ are positive.
  By Lemma~\ref{lem:subtrip}, we only need to check that 
  \begin{align*}
    a\in V_{b}^{*}(\gamma(t),\eta(t))\
    & ,\  b\in V_{a}^{*}(\gamma(t),\eta(t))\ ,\\
    \gamma(t)\in V^{*}_{\eta(t)}(a,b)\
    & ,\  \eta(t)\in V^{*}_{\gamma(t)}(a,b)\ ,\\
    \gamma(t)\in V_{a}(\eta(t),b)\
    & ,\ b\in V_{\gamma(t)}(a,\eta(t))\ ,\\
    \eta(t)\in V_b(\gamma(t),b)\
    & , \ a\in V_{\eta(t)}(b,\gamma(t))\ .
\end{align*}
Using again the fact that a diamond with extremities~$c$ and~$d$ is a connected component of $\Omega_c\cap\Omega_d$  (Proposition~\ref{pro:diam}),  
the  statement follows by continuity.
\end{proof}

Finally we also have as an immediate consequence of the connectedness of the positive cone:
\begin{lemma}[\sc Connectedness]\label{lem:path}
Let $a$ and $b$ be two transverse points.
\begin{enumerate}[leftmargin=*]
\item Assume $c$ is so that $(a,c,b)$ is positive.  Then there is
  $(\mathsf{H}, C)$ in~$\mathcal{H}$ such that $a$ and~$b$ belong to~$C$, and there is a path $t\mapsto c(t)$ from
  $[0,1]$ to $V_c(a,b)$ connecting $c=c(0)$ to $c(1)$ so that $(a,c(1),b)$ is
  a positive triple on $C$.
\item Assume furthermore that $d$ belongs to $V^*_a(c,b)$ then there are $(\mathsf{H}, C)$ in~$\mathcal{H}$, a
  path $t\mapsto c(t)$ as in the previous item, and a path $t\mapsto d(t)$ from
  $[0,1]$ to $V_c(a,b)$, so that $d(t)\in V^*_a(c(t),b)$ and $(a,c(1),d(1),b)$
  are on $C$.
\end{enumerate}
\end{lemma}
\begin{proof} Using a U-pinning at~$b$, we identify $\ms N$ with a
  positive semigroup $\ms N_b$ in $\ms U_b$ such that we have $V_c(a,b)=\ms
  N_b\cdot a$. The first point just follows from the connectedness of the
  positive semigroup~$\ms N_b$. For use in the second point we take a path $t
  \mapsto c(t)$ which is constant for~\mbox{$t>1/2$}.

Recall that $d=m_0\cdot c$, with $m_0\in\ms N_b$.
Let us define,  for $t\in [0,1/2]$,
$$d(t)=m_0\cdot c(t)\ ,$$
then we have by the
 semigroup property $d(t)\in V^*_{a}(c(t),b)$. Observe also that  $d(0)=d$. Then
 for $t\in [1/2,1]$, we have $c(t)=c(1/2)$, and we choose, using that $C$~contains
 elements of $V^*_a(c(1/2),d)$ (\emph{cf.}\ Proposition~\ref{pro:pref-circle}.(\ref{item3:pro:pref-circle})), a path $t\mapsto d(t)$ with  $d(t)\in 
V^*_a(c(1/2),d)$, and such that $d(1)$  belongs to $C$. \end{proof}

\subsubsection{Proof of the combinatorial properties}\label{sec:basic}

\begin{proof}[Proof of item~\eqref{it:basic1} of Proposition~\ref{pro:basic}]
  Assume $(a,b,c)$ satisfies the hypothesis. Let $(\mathsf{H}, C)$
  in~$\mathcal{H}$ and
  $t\mapsto c(t)$   obtained in Lemma~\ref{lem:path}. On~$C$,   a triple is positive if and only if the three points are pairwise distinct,
  the result thus follows from Lemma~\ref{lem:open-trip}.
\end{proof}

\begin{proof}[Proof of item~\eqref{it:basic3} of Proposition~\ref{pro:basic}]
  The ``only if'' part follows from the definition. Then for the ``if'' part
  we find, by  Lemma~\ref{lem:path}, $(\mathsf{H}, C)$ in~$\mathcal{H}$ and
  paths $t\mapsto x(t)$ and $t\mapsto y(t)$
  in $V_{x_0}(a,b)$ and  $V_{y_0}(a,b)$ respectively, such that
  $(x(0),y(0))=(x_0,y_0)$ and  $x(1)$, $y(1)$ are on~$C$, the $\ms H$-circle passing
  through~$a$ and~$b$. Then $(a,x(1),b,y(1))$ is positive and so is $(a,x_0,b,y_0)$ by Lemma~\ref{lem:open-quad}, since $x(t)$ and $y(t)$ are transverse thanks to Proposition~\ref{prop:opp-diamonds}.
\end{proof}

\begin{proof}[Proof of item~\eqref{it:basic2} of Proposition~\ref{pro:basic}]
  From the connectedness Lemma~\ref{lem:path} we obtain $(\mathsf{H},C)$
  in~$\mathcal{H}$ and paths $t\mapsto c(t)$,
  $t\mapsto d(t)$ such that $a,c(t),d(t),b$ are pairwise transverse, $c(0)=c$,
  $d(0)=d$, $(a,c(1),d(1),b)$ on~$C$ and $d(1)\in V_{a}(c(1),b)$. In particular $(a,c(1),d(1),b)$ is positive and thus by the deformation Lemma~\ref{lem:open-quad}, $(a,c,d,b)$ is positive.
\end{proof}

\begin{proof}[Proof of item~\eqref{it:basic4} of Proposition~\ref{pro:basic}]
  This is an immediate consequence of item~\eqref{it:basic2} and the fact that
  in order to check the positivity of a configuration one only needs to check
  the positivity of subtriples and subquadruples.
\end{proof}

\begin{proof}[Proof of item~\eqref{it:basic5} of Proposition~\ref{pro:basic}]
  If $(a,b,c,d)$ is positive, we have the strict inclusion $V_a(b,d)\subset
  V_a(c,d)$ and if $(a,c,b,d)$ is positive, 
we have the strict inclusion $V_a(c,d)\subset V_a(b,d)$. Hence a contradiction.
\end{proof}

\begin{proof}[Proof of item~\eqref{it:basic6} of Proposition~\ref{pro:basic}]
  Assume that $(a,x_0,b,x_1)$ is positive. Then $V= V_{x_0}(a,b)$ and $
  V_{x_1}(a,b)$ are  opposite diamonds. If both $(a,x_1,b,x_2)$ and $(a,x_0,b,x_2)$ are positive then we get that $x_2$ belongs to both $V$ and $V^*$, which is a contradiction.
\end{proof}

\begin{proof}[Proof of the necklace property (Corollary~\ref{coro:necklace})]  
Let us first remark that from item~\eqref{it:basic4} of
Proposition~\ref{pro:basic}, applied three times, (and using cyclic invariance
of positivity) the configuration
$$
(a,\gamma, b,\alpha, c,\beta)\ ,
$$
is positive. Thus $(\gamma,\alpha,\beta)$ is positive. \end{proof}

\subsection{Inclusion of diamonds}
 
\begin{proposition}[\sc Boundedness property]\label{pro:bounded}   Let $(a,b,d)$ be a positive triple and let $c\in V_b(a,d)$.  Assume that there
exist sequences  $\seq{b}$ and $\seq{c}$, converging respectively to $b$ and
$c$ and such that, for all~$m$, $(a,b_m,c_m,d)$ is  a positive
quadruple. Then   the sequence  $(\overline{V}{}^{*}_d(b_m,c_m))$
converges in the Hausdorff topology and  
\begin{equation*}
 \lim_{m\rightarrow\infty}\left(\overline{V}{}^{*}_d(b_m,c_m)\right)\ \subset  V_c(a,d)\ .
\end{equation*}
\end{proposition}

In particular,
\begin{coro}[\sc Inclusion]\label{coro:inclus}\label{coro:Bd1}
Let $(a,b,d,e)$ be a positive quadruple in~$\gp$. Then
\begin{equation*}
\overline{V}{}^{*}_e(b,d)\subset V_b(a,e)\ .
\end{equation*}
  
Let $(a,b,c,d,e)$ be a positive quintuple of points in~$\gp$.
then 
\begin{equation*}
\overline{V}{}_c(b,d)\subset V_c(a,e).
\end{equation*}
\end{coro}

Proposition~\ref{pro:bounded} will be proved in
Section~\ref{sec:proof-bound-prop}. Corollary~\ref{coro:Bd1} will be proved in
Section~\ref{sec:prel-circl} as the consequence of an intermediate statement.

\subsubsection{Preliminaries on circles}
\label{sec:prel-circl}

Let 
$V(a,d)$ be a diamond with extremities~$a$ and~$b$, and let $(\mathsf{H},C)$
be in~$\mathcal{H}$ such that $a$
and~$d$ belong to~$C$ and $C\cap V(a,d)\neq \emptyset$.

\begin{itemize}[leftmargin=*]
\item Let $
\delta=\{\delta{_t}\mid t\in\mathbb R\}$, 
be  the $1$-parameter group in $\ms H$ for which $d$~is the attracting fixed
point and $a$~is the repelling fixed point for the element $\delta_t$
($t>0$). The corresponding basins of attraction/repulsion are denoted by
\begin{align*}
O^+&\defeq \{x\in\gp \mid \lim_{t\to\infty}\delta_t(x)=d\}\ ,\\
O^-&\defeq \{x\in\gp \mid \lim_{t\to-\infty}\delta_t(x)=a\}\ .
\end{align*}
\item Let $\gamma=C\smallsetminus\{a,d\}$.
\item  Let $F$ be the set of fixed points of~$\delta$ in~$\gp$. 
 \end{itemize}
The result of this section is 
\begin{proposition}
For any $e$ in $\gamma$, we have $\overline{V}{}^{*}_a(e,d)\subset  \Omega_a$.  \label{pro:Bd1}
\end{proposition}
This proposition implies Corollary~\ref{coro:Bd1}:

\begin{proof}[Proof of the Corollary~\ref{coro:Bd1}]
	Applying Proposition~\ref{pro:Bd1}, we get that 
        \begin{equation*}
	\overline{V}{}^{*}_a(b,d)\subset \overline{V}{}^{*}_a(b,e)\cap \overline{V}{}^{*}_e(d,a)\subset \Omega_e\cap\Omega_a\ . 
      \end{equation*}
	Since $V_d({a,e})$ is a connected component of $\Omega_e\cap\Omega_a$ containing $V^*_a(b,d)$ it follows that 
        \begin{equation*}
\overline{V}{}^{*}_a(b,d)\subset V_{d}(a,e)\ .
\end{equation*}

This proves the first part of the corollary.

Suppose now that  $(a,b,c,d,e)$ is positive. Then the equalities
$V_b(a,e)=V_c(a,e)$ and $V_c(b,d) = V_{e}^{*}(b,d)$ together with the first
part imply the inclusion $\overline{V}{}^{*}_e(b,d)\subset
V_c(a,e)$.
\end{proof}
Recall that $S_a$ is complementary to $\Omega_a$. In order to prove Proposition~\ref{pro:Bd1},
we introduce the following sets for~$e$ in~$\gamma$
\begin{align*}
  J(e)&{\defeq}\overline{V}{}^{*}_a(e,d)\cap S_a\ , \\
  F(e)&{\defeq}J(e)\cap F\ .
\end{align*}
We will first prove that the sets  $J(e)$ and $F(e)$ are empty.
We first prove the following lemma

\begin{lemma}\label{lem:Bd2}
    Let $\gamma_1$ be one of the two connected components of~$\gamma$.
For any $e$ in $\gamma_1$,
\begin{enumerate}
\item $J(e)$ is invariant by the semigroup 
$
\delta^+{\defeq}\{\delta_t\mid t>0\}
$, {\it i.e.}\ $\delta_t( J(e))\subset J(e)$ for all~$t>0$,
\item \label{it:Bd2-3} $F(e)$ is independent of the choice of $e$ in~$\gamma_1$,
\item if $J(e)$ is not empty, so is $F(e)$,\label{it:Bd2-1}
\item \label{it:Bd2-2}     for all $c$ and $b$  in $\gp$ such that $(a,c,d,b)$ is a positive quadruple and 
$$
V_{a}^{*}(c,d)\cap\gamma_1\neq \emptyset \ ,
$$
 then
\begin{equation*}
F(e)\subset S_c\cup S_b\ .
\end{equation*}
 
\end{enumerate}
\end{lemma}

\begin{proof}
We prove the first point. By Proposition~\ref{pro:conf-circ} for  $t>0$
\begin{equation*}
V^*_a(\delta_t(e),d)\subset V^*_a(e,d)\ .
\end{equation*}
This implies that $V^*_a(e,d)$ is invariant by $\delta^+$ and so is $J(e)$.

The second point is 
 a consequence that $F$~is pointwise fixed by~$\delta_t$:
\begin{equation*}
  J(\delta_t(e))\cap F=\delta_t(J(e))\cap F=J(e)\cap F\ .
\end{equation*}

The third point is a consequence of the hyperbolic nature of the subgroup~$\delta$:
indeed, using a linear representation of~$\G$ we can
assume that $\delta$~is a one-parameter subgroup of diagonal matrices and that
$J(e)$~is a non-empty closed  subset of the projective space invariant by the
semigroup~$\delta^+$; in this case, every ray $( \delta_t ( x))_{t\geq 0}$
(for~$x$ in $J(e)$)
has a limit that is a fixed point of~$\delta$.

Let us prove now the last point.
Let $(a,c,d,b)$ be as in the hypothesis.  Thus 
$V_{a}^{*}(c,d)\cap \gamma_1\not=\emptyset$ and by point~(\ref{it:Bd2-3}) we can
choose~$e$ in this intersection. By
Proposition~\ref{pro:basic}.(\ref{it:basic4}), $(a,c,e,d,b)$ is a positive
configuration and hence $V^{*}_{a}(e,d) \subset V^{*}_{a}(c,b)$. From this we get 
\begin{equation}
F(e)\subset J(e)=\left(\overline{V}{}^{*}_a(e,d)\cap S_a\right)\subset\left( \overline{V}{}^{*}_a(c,b)\cap S_a\right)\ . \label{it:Bd1-1}
\end{equation}
Now, we remark that since $(a,c,d,b)$ is positive, we have  $$
V^{*}_{a}(c,b)\subset\Omega_a.$$
Thus 
\begin{equation}
\overline{V}{}_a^{*}(c,b)\cap S_a \subset \overline{V}{}_a^{*}(c,b)\smallsetminus V_{a}^{*}(c,b).\label{it:Bd1-2}
\end{equation}
But since $V^{*}_{a}(c,b)$ is a connected component of  the open set $$\Omega_c\cap\Omega_b=\gp \smallsetminus \left(S_c\cup S_b\right)\ ,$$
we get 
\begin{equation}
\left(\overline{V}{}^{*}_a(c,b)\smallsetminus V^{*}_{a}(c,b)\right)\subset \left(S_c\cup S_b\right). \label{it:Bd1-3}
\end{equation}
Combining inclusions~\eqref{it:Bd1-1}, \eqref{it:Bd1-2}  and~\eqref{it:Bd1-3},
we get that
\begin{equation*}
F(e)\subset \left(S_c\cup S_b\right)\ .\qedhere
\end{equation*}
\end{proof}

We can now prove Proposition~\ref{pro:Bd1}, in other words that  $J(e)$~is empty. By item~\eqref{it:Bd2-1} of Lemma~\ref{lem:Bd2}, it suffices to show that $F(e)$~is empty. The fact that $F(e)$ is empty follows from item~\eqref{it:Bd2-2} of Lemma~\ref{lem:Bd2} and the following result.

\begin{lemma}
Let $Q$ be a  subset of~$\gp$. Assume that there exist nonempty open sets~$U$
and~$V$ such that for all $c$ in $U$, and for all $b$ in $V$,
\begin{equation*}
Q\subset S_c\cup S_b,
\end{equation*}
then $Q$ is empty.
\end{lemma}
\begin{proof}  Let~$q$ be in~$Q$ and set $Z\defeq S_q$. Then $Z$~is a proper closed Zariski subset of~$\gp$. Observe that if $u\not\in Z$, then
\begin{equation*}
  q\notin S_u.
\end{equation*}
On the other hand we can find~$u$ in the nonempty set $U\smallsetminus Z$ and~$v$ in the nonempty set $V\smallsetminus Z$, and by hypothesis $q\in S_u\cup S_v$. This shows that $q\neq q$
and concludes the proof.
\end{proof}

\subsubsection{Proof of the boundedness Proposition~\ref{pro:bounded}.}
\label{sec:proof-bound-prop}

  We use the notation of the previous paragraph.

  First note that $\overline{V}{}^{*}_a(b_m,c_m) =
  \overline{V}{}^{*}_d(b_m,c_m) = \overline{V}{}^{*}_a(b_m,d) \cap
  \overline{V}{}^{*}_d(a,c_m)$ (cf.{} Corollary~\ref{coro:inter-diam}). Since the sequences $\seq{b}$ and $\seq{c}$
  converge, the sequences of closures of diamonds $(\overline{V}{}^{*}_a(b_m,d))
  $ and $(\overline{V}{}^{*}_d(a,c_m))$ also converge and thus the sequence
  $(\overline{V}{}^{*}_a(b_m,c_m))$ converges.  
  Let~$C$ be an $\ms H$-circle through~$a$
and~$d$ such that $\gamma\defeq V_c(a,d)\cap C$ is not empty. Since being positive is an open condition for quadruples, we can find $e$ and $f$ in $\gamma$ so that 
$(e,b_m,c_m,f)$ is positive for $m$ large enough as well as $(a,e,f,d)$. Thus
$$
V^{*}_{a}(b_m,c_m)\subset V^{*}_{a}(e,f)
$$
 Applying  Proposition~\ref{pro:Bd1}, we get that 
$$
\overline{V}{}^{*}_a(b_m,c_m)\subset \overline{V}{}^{*}_a(e,f)\subset V_c(a,d)\ ,
$$
which easily implies the result.

\subsection{Left and right limits of positive maps}
Our main  result is 
\begin{proposition}[\sc Existence of left and right limits]\label{pro:ext}
Let $S$ be a totally ordered set and $\phi$ be a positive map from~$S$ to~$\gp$. 

Let $\{b_n\}_{n\in\mathbb N}$ be a sequence 
in $S$ such that there exist $a$, $b$, and $c$ in~$S$ with $a<
b_n\leq  b_{n+1}\leq b<c$, for all~$n$. 
Then $\{\phi(b_n)\}_{n\in\mathbb N}$ converges to a point $y$ in $V_b(a,c)$.

Symmetrically, let $\{a_n\}_{n\in\mathbb N}$ be a sequence of points such that $c<a\leq a_{n+1}\leq  a_n<b$.  Then $\{\phi(a_n)\}_{n\in\mathbb N}$ converges to a point $y$ in  $V_a(c,b)$. 
\end{proposition}

\begin{rema}
  In the case of Lusztig's total positivity, this statement was proved in
  \cite[Theorem 7.4]{Fock:2006a}, this is also proved for $\mathsf{SL}_n(\mathbb R)$ in \cite[Section 5]{Labourie:2006}.
\end{rema}

As an immediate  corollary, we show that positive maps defined on dense subsets extend to positive maps. More precisely:

\begin{coro}[\sc Extension of positive maps]\label{coro:posmap}
Let $A$ be dense subset in $[0,1]$. Assume that we have a positive map $\xi$ from $A$ to $\gp$. Then there exist
\begin{itemize}
\item 	a unique left-continuous positive map $\xi_+$  from $[0,1]$ to $\gp$
  such that $\xi$ coincide with $\xi_+$ on a dense subset of $A$, 
\item a unique right-continuous positive map $\xi_-$  from $[0,1]$ to $\gp$ such that $\xi$ coincide with $\xi_-$ on a dense subset of $A$.
\end{itemize}
Moreover,
\begin{itemize}
\item for any ordered quadruple $(x,y,z,t)$ of pairwise distinct points in $[0,1]$
$$
(\xi_\epsilon(x),\xi_\eta(y),\xi_\nu(z),\xi_\beta(t))\ ,
$$
is a positive quadruple for any choice of $\epsilon$, $\eta$, $\nu$, and $\beta$ in $\{+,-\}$,
\item if $\seq{x}$, $\seq{z}$ are sequences in $[0,1]$ converging to $y$, with
  for all~$n$, $x_m<y<z_m$, then
 $$
\lim_{m\to\infty}\xi_{\epsilon_m}(x_m)=\xi_+(y)\ , \ \lim_{m\to\infty}\xi_{\eta_m}(z_m)=\xi_-(y)\ ,
$$
for any sequences $\seq{\epsilon}$ and $\seq{\eta}$ in $\{+,-\}$.
\end{itemize}

\end{coro}
\begin{proof}[Proof of Proposition~\ref{pro:ext}] Let us define $x_n=\phi(b_n)$.
We can write $x_n=m_n\cdot x_{n-1}$, with $m_n\in {\ms N}_{\phi(c)}$. Thus by induction we have 
\begin{equation*}
x_n=m_n\cdots m_1\cdot x_0\ .
\end{equation*}

But we know that $V^{*}_{\phi(c)}(x_0, \phi(b))$ is a relatively compact region of
$\Omega_{\phi(c)}$ by Proposition~\ref{pro:bounded}. Thus
$\{x_n\}_{n\in\mathbb N}$ is relatively compact in $\Omega_{\phi(c)}$. It follows that
\begin{equation*}
\pi_n=m_n\cdots m_1,
\end{equation*}
is a bounded sequence in ${\ms N}_{\phi(c)}$. We now prove that this sequence converges. Assume that we have subsequences that converge to different limits~$u$ and~$v$. After extracting further subsequence, we may find subsequences
\begin{equation*}
q_i=\pi_{n_i},\ \ p_i=\pi_{m_i},\hbox{ with } n_i\leq m_i
\end{equation*}
such that $\{q_i\}_{i\in\mathbb N}$ converges to $u$ and $\{p_i\}_{i\in\mathbb
  N}$ converges to $v$. It follows that $u=w_1\cdot v$ with $w_1\in
\overline{\ms N}_{\phi(c)}$. Symmetrically, $v=w_0\cdot u$ with
$w_0\in\overline{\ms N}_{\phi(c)}$. It follows that $w_0\cdot w_1=1$, thus
$w_0$ and $w_1$ are invertible in the closed semigroup~$\overline{\ms N}_{\phi(c)}$, hence equal to the identity.  In particular $u=v$ and  $\{x_n\}\inn$ converges.

The proof for the sequence $\{a_n\}_{n\in\mathbb N}$ is symmetric.
\end{proof}

\subsubsection{Positivity and continuity}
In some cases, it suffices to show that the image of every triple is positive

\begin{proposition}[\sc Triples suffice]\label{pro:triplenough} Let $\phi$ be a continuous map from  an interval   $I$ to $\gp$ such that the image of every ordered triple is positive, then $\phi$ is positive.
\end{proposition}
\begin{proof}  
  Since $[x,y]$ is connected and $\phi$~is continuous,
  $V_{\phi(z)}(\phi(x),\phi(y))$ does not depend on $z$ in $]x,y[$, and we
  denote it $W(x,y)$.

Let $V(t,y)$ be the unique diamond with extremities $\phi(t)$ and $\phi(y)$ obtained in Lemma~\ref{lem:semigroup} so that 
\begin{equation}
	V(t,y)\subset W(x,y)\ .\label{eq:proof-VinW}
\end{equation}
\vskip 0.2 truecm
\noindent{\sc Main step:}
We first prove that if $x<t<y$, then 
\begin{equation}
	W(t,y)=V(t,y)\ . \label{eq:proof-triplenough1}
\end{equation}
Let us consider 
$$
U=\{t\in ]x,y[\mid W(t,y)=V(t,y)\}
$$
Let us write $W(x,y)=\ms N \cdot \phi(x)$, where $\ms N$ is an open  semigroup in
$\ms U_{\phi(y)}$. We can thus write $\phi(t)=n_t\cdot \phi( x)$, with $t\mapsto
n_t$ a continuous map defined on $]x,y[$ with values in $\ms N$; the limit
of~$n_t$ when $t$~tends to~$x$ is equal to $\id$.  Then  $V(t,y)=n_t W(x,y)$. We now proceed to the proof and show that $U$~is open, nonempty and closed.

\begin{enumerate} 
\item The set $U$ is also the set of $t$ for which there exists~$s$, with $t<s<y$ such that $\phi(s)$ is in $V(t,y)$. In other words,  $n_t^{-1}n_{s}$ belongs to $\ms N$. Thus  $U$ is open. 
\item  Since $\ms N$ is open, given  $s$, for all  $t$ close enough to $x$ we have  $n_t^{-1} n_s$ is in $\mathsf{N}$. Thus $n_s\in n_t\ms N$, hence $\phi(s)\in V(t,y)$. Therefore $U$ is nonempty. 
\item Let~$t$ be in the closure of~$U$. Let $s$ be $>t$ and let $\seq{t}$ be a
  sequence in~$U$ converging to~$t$. Thus $\seqm{n_{t_m}^{-1} n_s}$ converges
  to $n_{t}^{-1} n_{s}$. Since, for $m$ big enough, we have $s>t_m$, the
  element $n_{t_m}^{-1} n_s$ belongs to~$\ms N$; hence $n_{t}^{-1} n_{s}$ belongs
  to~$\overline{\ms N}$, \emph{i.e.}\ $\phi(s)$ belongs to $\overline{V}(t,y)$. As the map~$\phi$ is transverse, $\phi(t)$ is transverse
  to~$\phi(s)$ and this implies that $\phi(s)$ belongs to $V(t,y)$ and $n_{t}^{-1} n_{s}$ belongs
  to~${\ms N}$. Therefore $t$~belongs to~$U$, and we have completed the proof that $U$ is closed.
\end{enumerate}
The proof of the Equation~\eqref{eq:proof-triplenough1} is now complete.

\vskip 0.2 truecm
\noindent{\sc Conclusion:}  Let $(a,b,c,d)$ so that $a<b<c<d$, with all subtriples of $(\phi(a),\phi(b),\phi(c),\phi(d))$ positive. By item~\eqref{it:basic2} of Proposition~\ref{pro:basic}, we only have to prove that 
$$
V_{\phi(d)}(\phi(a), \phi(b))=V^{*}_{\phi(c)}( \phi(a), \phi(b))\ ,
$$
Observe that 
$$W(a,b)\subset W(a,d)=V_{\phi(b)}(\phi(a),\phi(d))\ ,$$
by Equation~\eqref{eq:proof-triplenough1}.  Thus $\phi(d)$ does not belong to $W(a,b)$ and hence belongs to $W^{*}(a,b)$ by Lemma~\ref{lem:semigroup}. We thus  have 
  $$
V_{\phi(d)}(\phi(a),\phi(b))=W^*(a,b)=V^{*}_{\phi(c)}(\phi(a),\phi(b))\ .$$
This completes the proof that the  quadruple $(\phi(a),\phi(b),\phi(c),\phi(d))$ is positive, hence of the proposition. 
 
\end{proof}

\subsection{Proximal elements}

In this section we show that positive equivariant maps give rise to proximal elements. 

We first prove the following proposition for elements in $\G$: 
\begin{proposition}
  Let~$g$ be in~$\G$ and let $g^-$, $g^+$, and~$a$ be in~$\gp$
  such that $g^-$ and~$g^+$ are fixed by the action of~$g$ and
  that the quadruple $(g^-, a, g\cdot a, g^+)$ is positive.

  Then the action of~$g$ on~$\gp$ is proximal, its attracting fixed point
  is~$g^+$, and its repelling fixed point is~$g^-$.
\end{proposition}
\begin{proof}
  Up to the action of $\operatorname{Aut}_0(\mk g)$ we can assume that
  $g^-$ is the point in~$\gp$ with stabilizer equal
  to~$\mathsf{P}_\Theta$, $g^+$ is the point with stabilizer
  $\mathsf{P}_{\Theta}^{\mathrm{opp}}$, and that $a=n\cdot g^+$ with $n$
  in~$\mathsf{N}\subset \mathsf{U}_\Theta$. One then has $g$ belonging
  to~$\mathsf{L}_\Theta$ and $g\cdot a = n'\cdot g^+$ with $n'$ in
  $\mathsf{N}$ equal to $g ng^{-1}$.

  For every~$\alpha$ in~$\Theta$, we denote by $\pi_\alpha\colon
  \mathfrak{u}_\Theta\to \mathfrak{u}_\alpha$ the $\mathsf{L}_\Theta$-equivariant
  projection and by $C_\alpha$ the salient invariant open convex cone
  in~$\mathfrak{u}_\alpha$ defining positivity (cf.\
  Remark~\ref{rem:detail_on_cones}). 

  Let~$\alpha$ be in~$\Theta$. The positivity of the quadruple  $(g^-, a,
  g\cdot a, g^+)$ implies that the elements $x=\pi_\alpha( \log n)$
  and $x'=\pi_\alpha( \log n')$ are both in~$C_\alpha$ and that $x-x'$ also
  belongs to~$C_\alpha$.  Let $A$  be the automorphism of $\mathfrak{u_\alpha}$ given by the restriction of the adjoint action of~$g$ to the
  subspace~$\mathfrak{u}_\alpha$; 
 one has thus $A(C_\alpha)=C_\alpha$ and
  $A(x)=x'$. This implies that the set $B\defeq (-x+C_\alpha)\cap
  (x-C_\alpha)$ is sent into $(-x'+C_\alpha)\cap
  (x'-C_\alpha)$ by~$A$ (where $x-C_\alpha =\{ x-y \mid y\in C_\alpha\}$). Therefore the element~$A$ is contracting for the norm
  on~$\mathfrak{u}_\alpha$ whose unit ball is the open convex set~$B$.

  We deduce from this that all the eigenvalues of the adjoint action
  of~$g$ on $\bigoplus_{\alpha\in \Theta} \mathfrak{u}_\alpha$ are of
  modulus less than $1$. Since this subspace generates~$\mathfrak{u}_\Theta$ \cite{Kostant_RootLevi}, we have also
  that all the eigenvalues of~$g$ on~$\mathfrak{u}_\Theta$ are of modulus
  $<1$. This means precisely that $g^+$ is an attracting fixed point for
  the action of~$g$ on~$\gp$ and thus $g$~is proximal. For the same
  reasons, $g^-$ is the repelling fixed point of~$g$.
\end{proof}

From this, we immediately get:

\begin{proposition}\label{pro:prox-pos}
  Let~$\gamma$ be a homeomorphism of the circle~$S^1$ having one attracting
  fixed point~$\gamma^+$ and one repelling fixed~$\gamma^-$ in~$S^1$. Let
  $S\subset S^1$ be an infinite $\gamma$-invariant set containing $\gamma^+$
  and~$\gamma^-$ and
  let $\xi$ be positive map from~$S$ to~$\gp$. Assume that there exists an element~$g$ in~$\G$ such
  that, for all~$s$ in~$S$,
  \[
    \xi(\gamma(s))=g\cdot \xi(s)\ ,
  \]
  Then $g$~is proximal and $\xi(\gamma^+)$ is the attracting fixed point
  of~$g$ and similarly $\xi(\gamma^-)$ is the repelling fixed point of~$g$.
\end{proposition}

\section{Triples, Tripods and metrics}\label{sec:metric}

In this section,  we construct  for every positive triple $(a,b,c)$  a
complete metric on  the diamond $V_c(a,b)$ (Definition~\ref{def:diammet}). We
also show that this family of metrics satisfies contraction properties (Propositions~\ref{pro:Acontract2} and~\ref{pro:contract2}).

We first do it for triples of a special type that we call tripods.

\subsection{Tripods and metrics}

Recall that in Proposition~\ref{pro:pref-circle},
we fixed~$\mathcal{H}$ a class of pairs $(\mathsf{H}, C)$ where the subgroups  $\mathsf H$
are isogenic to $\psld$ and $C$ is an  $\mathsf H$-orbit, isomorphic to~$\rp$
and called an {\em $\ms H$-circle}.

\begin{definition} A {\em tripod} is a triple of pairwise distinct points
  on~$C$ for some $(\mathsf{H},C)$ in~$\mathcal{H}$.
\end{definition}
 A tripod is always positive. If $\tau=(x,t,y)$ is a tripod, we write  
 \[
  \tau^-=x\, ,\  \tau^0=t\, , \ \tau^+=y\, ,\ \overline\tau\defeq(y,t,x)\, ,\  V_\tau\defeq V_{\tau^0}(\tau^-,\tau^+)\, .
\] 
More generally, for a positive triple $t=(a,b,c)$, we write $t^-=a$, $t^0=b$,
$t^+=c$ and $V_t\defeq V_b(a,c)$.

Let $\mathcal T_0$ be the set of tripods. Observe that $\operatorname{Aut}_0(\mk
g)$ acts transitively on the left  on the space of
tripods, and that the positive circle containing a tripod is unique.

The stabilizer of any
tripod is compact (\emph{cf.}\ Section~\ref{sec:trip-param} below),
   in particular $\operatorname{Aut}_0(\mk g)$ acts properly on the space of
tripods, and we can then fix~$d$ an $\operatorname{Aut}_0(\mk g)$-invariant Riemannian metric on the set of tripods~$\mathcal T_0$.

\subsubsection{Tripods and the parametrization}
\label{sec:trip-param}

Let us consider (as 
in Paragraph~\ref{sec:param}),
the convex cone~$\mathsf{C}$ and 
the   parametrization $\Psi\colon\ms C\to \ms N$ equivariant with respect to $\mathsf{L}^{\circ}_{\Theta}$
 ($\Psi$~is
given by the product of exponential maps). Note that $\Psi$~extends continuously
to a map~$\overline{\mathsf{C}}\to \mathsf{U}_\Theta$ that is also $\mathsf{L}^{\circ}_{\Theta}$-equivariant.

Let $h$ be the element of $\ms  C$ corresponding to the unipotent associated
to the preferred $\sld$ --- see the proof of Proposition~\ref{pro:pref-circle}. Let $\ms K_h$
be the stabilizer of $h$ in $\ms L^{\circ}_{\Theta}$. Since the stabilizer of a positive triple is compact, it follows that $\ms K_h$~is compact.

If now $x$ and $y$ are transverse points in $\gp$ and $\sigma\colon
\mathsf{U}_\Theta\to \mathsf{U}_y$ is  a U-pinning at~$y$, then the map
\begin{equation*}
\Psi^\sigma \colon \mathsf{C}\mapsto \gp, \  u\mapsto \sigma\circ\Psi(u)\cdot x\, ,
\end{equation*}
is a parametrization of the diamond $V_t(x,y)$ with $t\defeq \Psi^{\sigma}(h)$.
 We then define

\begin{definition}
 Given a tripod $\tau=(x,t,y)$ a {\em $\tau$-parametrization of the diamond} $V_\tau$, is a map $\Psi_\tau$ of the form $\Psi^\sigma$ so that $\Psi^\sigma(h)=t$.
\end{definition} 

From the definition follows

\begin{proposition}\label{pro:paramkh}
 Given a tripod $\tau$, a $\tau$-parametrization of the diamond exists and is unique up to postcomposition by the stabilizer of $\tau$ (equivalently up to precomposition by $\ms K_h$).
\end{proposition}

The next proposition is crucial; it insures that a sequence of
parametrizations of diamonds associated with tripods converges to the
constant map as soon as one sequence in the image converges, precisely
\begin{proposition}[\sc Contraction in corners]\label{pro:dm-contract}

Let $\seq{\tau}$ be a sequence of tripods, with $\tau_m=(x_m,t_m,y)$, and $\Psi_{\tau_m}$ a $\tau_m$-parametrization for all $m$. 

Assume that $\seq{x}$ converges to a point $x$ transverse to $y$. Assume
that there exists a converging sequence $\seq{k}$ in  the cone $\ms C$, and
such that
\begin{equation}
  \lim_{m\to\infty}\Psi_{\tau_m}(k_m)=x\ .
\end{equation}
Then for any sequence $\seq{k'}$ in $\ms C$ that is  bounded in~$\overline{\mathsf{C}}$, 
\begin{equation}
  \lim_{m\to\infty}\Psi_{\tau_m}(k'_m)=x\ .
\end{equation}
\end{proposition}

\begin{proof}
  Since $x$ is transverse to $y$, by replacing $\tau_m$ by $u_m\tau_m$ where $\seq{u}$ is a converging sequence in $\ms U_y$, we may as well assume that $\seq{x}$ is constant and equal to~$x$.

Using the fact that $\operatorname{Aut}_0(\mk g)$ acts transitively on
tripods, let us  write  $t_m=g_m\cdot t_0$, with $g_m$ fixing~$x$ and~$y$. 
Thus
\begin{equation*}
\Psi_{\tau_m}(h)= g_m \cdot t_0= g_m\cdot \Psi_{\tau_0}(h) =g_m \sigma( \Psi(h))\cdot x \ .
\end{equation*}
Note that the U-pinning $\sigma\colon \mathsf{U}_\Theta\to \mathsf{U}_y$ comes from
an element~$\sigma$ of~$\operatorname{Aut}_0(\mk g)$.
Denote $g_{m}^{0} = \sigma^{-1}\circ g_m \circ \sigma$; it is
an element of~$\operatorname{Aut}_0(\mk g)$ stabilizing the standard
unipotent subalgebras $\mathfrak{u}_\Theta$ and $\mathfrak{u}_{\Theta}^{\mathrm{opp}}$.
Up to maybe
precomposing by an element of~$\mathsf{K}_h$, we may assume that
$\Psi_{\tau_m}$ is the map $k\mapsto \sigma( \Psi( g_{m}^{0}\cdot k)) \cdot x$.

Therefore we have, for any $\seq{\ell}$ in~$\overline{ \mathsf{C}}$, that $\seqm{
\Psi_{\tau_m}(\ell_m)}$ converges to $x$ if and only if the sequence $\seqm{
g_{m}^{0}\cdot \ell_m}$ converges to~$0$ in~$\overline{ \mathsf{C}}$.

For any~$y$ in~$\overline{\mathsf{C}}$, let 
\begin{equation*}
	K(y)\defeq \overline{ \mathsf{C}} \cap \bigl( y - \overline{\mathsf{C}}\bigr)\ .
\end{equation*}
Since $\overline{ \mathsf{C}}$ is salient, $K(y)$ is compact for any $y$. 

From the previous discussion, we  get that the sequence $\seqm{ c_m=g_{m}^{0}\cdot k_m}$ converges
to~$0$. Thus, using again the fact that $\overline{ \mathsf{C}}$ is salient, for every positive real~$R$, the sequence of compact sets $\seqm{K(R\cdot c_m)}$
converges to~$\{0\}$.

Let now $\seqm{ k_{m}^{\prime}}$ be a sequence in~$\mathsf{C}$,  bounded in~$\overline{\mathsf{C}}$. Since by hypothesis $\seqm{ k_{m}}$ converges  in~$\mathsf{C}$,  there exists a positive real 
 $R$ such that, for all~$m$, $R\cdot k_m -k_{m}^{\prime}$ belongs
 to~$\overline{\mathsf{C}}$. In other words: $k_{m}^{\prime}$ belongs to  $R\cdot k_{m} -\overline{\mathsf{C}}$. Thus, for all~$m$, $g_{m}^{0}\cdot k_{m}^{\prime}$
belongs to $K( R\cdot c_m )$. Hence the sequence $\seqm{ g_{m}^{0}\cdot
k_{m}^{\prime}}$ converges to~$0$. This means that  the sequence $\seqm{ \Psi_{\tau_m}(
k_{m}^{\prime})}$ converges to~$x$ as wanted.
\end{proof}

\subsubsection{Diamond metrics for tripods}

We choose once and for all a Euclidean distance~$d_0$  on the convex cone $\mathsf{C}$, associated with the
Riemannian~$g_0$ induced by a $\ms K_h$-invariant scalar product on $\mathfrak{u}_\Theta$. 
This distance~$d_0$ is $\ms K_h$-invariant and extends to~$\overline{ \mathsf{C}}$.

\begin{definition} 
Given a tripod $\tau=(x,t,y)$, let $\Psi_\tau$ be a $\tau$-pa\-ra\-me\-tri\-za\-tion of $V_\tau$, let 
\begin{equation*}
g^+_\tau\defeq (\Psi_\tau)_*g_0\ , \ \ g^-_\tau\defeq (\Psi_{\overline\tau})_*g_0\
, \ \ g_\tau=g_{\tau}^{+} + g_{\tau}^{-} \ ,
\end{equation*}
as well as  $d^+_\tau$, $d^-_\tau$, and $d_\tau$ the associated distances so that
$$
d^\pm_\tau\leq d_\tau\leq d^+_\tau+d^-_\tau\ .
$$ 
The metric $g_\tau$ is the {\em diamond metric} (for the tripod~$\tau$) on $V_\tau$,
while $d_\tau$ is the  {\em diamond distance}. 
\end{definition}

The terminology is justified by 
\begin{proposition}[\sc Uniqueness and Completeness]\label{pro:dm-complete}
  The diamond metric is independent of the choice of the $\tau$-parametrization and
  only depends on~$\tau$. Moreover, $d_\tau$ is complete and proper on
  $V_\tau$.

  There exists a function $F\colon \mathbb{R}_{>0} \to \mathbb{R}_{>0}$ with
  $\lim_0 F=1$ and such that the following holds:
  For any tripods~$\tau$ and $\tau'$ with the same extremities~$\tau_-$
  and~$\tau_+$, if $d_\tau( \tau^0, \tau^{\prime 0})\leq \varepsilon$
then
$$
 F(\varepsilon)^{-1}d_{\tau'} \leq d_{\tau}\leq F(\varepsilon)\  d_{\tau' } \ .
$$ 
 
\end{proposition}

\begin{proof} The independence on the parametrization is a consequence of Proposition~\ref{pro:paramkh} and the fact that $d_0$ itself is invariant under the group~$\ms K_h$. 

Let us now prove completeness. Let $\seq{u}$ be a Cauchy sequence for
$d_\tau$, then it is a Cauchy sequence for both $d^+_\tau$ and $d^-_\tau$. It
follows that the sequences $\seq{v}$ and $\seq{w}$ defined by
\begin{equation*}
  v_m=\Psi_\tau^{-1}(u_m) , \
  w_m=\Psi_{\overline\tau}^{-1}(u_m)\ ,
\end{equation*}
are both Cauchy sequences.      
Since $\overline{\mathsf{ C}}$ is complete with respect to the metric $d_0$, there exist~$v$ and~$w$ in~$\overline{\mathsf{C}}$ such that
\begin{equation*}
		\lim_{m\to\infty}v_m=v  , \ \lim_{m\to\infty}w_m=w\ .
\end{equation*}
 Since by construction $\Psi_{\tau}$ and $\Psi_{\overline\tau}$ extend continuously to the closure of~$\mathsf{C}$,
\begin{equation*}
	\Psi_{\tau}(v)=\Psi_{\overline\tau}(w)=\lim_{m\to\infty}u_m\eqdef u\ .
\end{equation*}
Obviously $u$ belongs to $\overline{V}_\tau$.
By construction $u=n_y\cdot x$ and $u=n_x\cdot y$ where $n_y$ belongs to $\ms U_y$ and  $n_x$ belongs to $\ms U_x$. Since $x$ and $y$ are transverse, it follows that $u$ is transverse to both $y$ and $x$. Thus $u$ belongs to $V_\tau$. Hence $d_\tau$ is complete.
\vskip 0.2 truecm
  Let us prove the last part.  Observe now by hypothesis, there
exists~$\ell$ in~$\ms L_\Theta$ such that $\Psi_{\tau'}=\Psi_\tau \circ \ell$. Thus
$$
g^+_{\tau'}=(\Psi_{\tau})_* (\ell_* g_0)\ . 
$$
Recall that $g_0$~is induced by a scalar product on~$\mk u_\Theta$, and that
$\ms L_\Theta$~acts linearly on~$\mk u_\Theta$, it then follows that there is a
function~$F_0$ on~$\ms L_\Theta$ such that
$$
F_0(\ell)g_0\leq \ell_*(g_0)\leq F_0(\ell) g_0\ ,
$$
and with $F_0(\ell)\underset{\ell\to \mathsf{K}_\Theta}{\longrightarrow} 1$
where $\mathsf{K}_\Theta$~is a maximal compact subgroup~$\mathsf{L}_\Theta$. 
Pushing forward by $\Psi_\tau$ we have  
$$
F_0(\ell)g^+_\tau\leq g^+_{\tau'}\leq F_0(\ell) g^+_{\tau}\ .
$$
The same holds for $g^{-}_{\tau}$ and $g^{-}_{\tau'}$. Hence the same
inequality holds for $g_\tau=g_{\tau}^{+}+g_{\tau}^{-}$ and $g_{\tau'}$; 
this concludes the proof with the remark that $d_\tau( \tau^0, \tau^{\prime
  0}) \longrightarrow 0$ implies that $\ell\longrightarrow \mathsf{K}_\Theta$. Precisely, we can define
  $$
  F(\epsilon)=\sup\{F_0(h)\mid d_\tau( \tau^0, h(\tau^{0}))\leq \epsilon\}\ ,
  $$
  and observing that by equivariance $F$ does not depend on the choice of $\tau$. 
\end{proof}

\begin{proposition}[\sc Contraction for tripods]\label{pro:Acontract}
Let $\seq{\tau}$ be a  sequence of tripods. Assume that, for all~$m$ in~$\mathbb{N}$, $V_{\tau_{m+1}}\subset V_{\tau_m}$ and that
\begin{equation}
\bigcap_{m\in\mathbb N}V_{\tau_m}=\{z\}\ .\label{eq:hypT}
\end{equation}
For any positive $R$, let $V_{\tau_m}(R)$ be the ball of radius $R$ and center $\tau^0_m$ with respect to $d_{\tau_m}$. Then on $V_{\tau_m}(R)$, we have 
\begin{equation*}
g_{\tau_0}\leq k_m \cdot g_{\tau_m}\ ,
\end{equation*}
with $\seq{k}$ converging to zero.

\end{proposition}
\begin{proof}
  Since $\operatorname{Aut}_0(\mk g)$ acts transitively on the space of
  tripods~$\mathcal T_0$, it follows that $\tau_m=h_m\cdot \tau_0$, for
  some~$h_m$ in $\operatorname{Aut}_0(\mk g)$.  Since the construction of the
  tripod metrics is $\operatorname{Aut}_0(\mk g)$-equivariant,  we observe
  that $g_{\tau_m}=h_{m}^{*}g_{\tau_0}$.
  Then, we take
  \begin{align}
    k_m&=\sup\left\{ \frac{g_{\tau_0}(w,w)}{g_{\tau_m}(w,w)}
         \mathrel{\Big|}  w\in \mathsf{T} V_{\tau_m}(R) \right\}\\
       &= \sup\left\{\frac{g_{\tau_0}(w,w)}{g_{\tau_0}(\ms T h_m^{-1}w,
                  \ms T h_m^{-1} w)} \mathrel{\Big|}  w\in \mathsf{T}
         V_{\tau_m}(R) \right\}\\
       &=\sup\left\{\frac{g_{\tau_0}(\ms T h_m
         (v),\ms T h_m (v))}{g_{\tau_0}(v,v)} \mathrel{\Big|}  v\in \mathsf T V_{\tau_0}(R) \right\}\ .\label{eq:Prox1}
  \end{align}
 The hypothesis~\eqref{eq:hypT} says that $\seq{h}$, seen as a sequence of
 diffeomorphisms of~$\gp$ converges uniformly on every compact set  of
 $V_{\tau_0}$ to the constant map. It follows that $\seq{h}$ also converges
 $C^1$ to the constant map on any compact set in $V_{\tau_0}$ and hence
 $\seq{\ms T h}$ converges to zero  uniformly on every compact set  of
 $V_{\tau_0}$. Thus, equality~\eqref{eq:Prox1} shows that $\seq{k}$ converges to zero.\end{proof}

\subsection{Positive triples, tripods and metrics}\label{par:met-trip}

Our goal is to construct a complete metric on the diamond associated with a
positive triple and to prove a generalization of the contraction properties
(Propositions~\ref{pro:Acontract2}  and~\ref{pro:contract2}).

\subsubsection{Approximating triples: the tripod defect}
We will first approximate in a rough sense positive triples by tripods. For
any positive triple $t=(x,z,y)$, let 
$$
{\mathrm K}(t)\defeq \inf\bigl\{d_\tau (z,\tau^0)\mid\tau\in\mathcal T_0\, ,\
z\in V_\tau\, , \ 
(\tau^-,\tau^+)=(x,y) \bigr\}\ .
$$
We call $\mathrm{ K}(t)$ the {\em tripod defect}.

Observe that $\mathrm{K}(t)$ depends continuously on $t$, and that the tripod defect vanishes for tripods. Let also
\begin{align*}
D(t,K_0)&\defeq \bigl\{ \tau\in\mathcal T_0 \mid (\tau^-,\tau^+)=(x,y)\ ,
          d_\tau(z,\tau^0)\leq K_0 \bigr\}\ ,\\
D(t)&\defeq D(t,\mathrm{K}(t))\ .
\end{align*}

\begin{proposition}\label{pro:diameter} 
\begin{enumerate}
\item 	Given $K_0\geq {\mathrm{K}(t})$,  the set $D(t,K_0)$ is compact and non-empty. 
\item ${\mathrm{ K}(t)}=0$ if and only if $t$ is a tripod. 
\item\label{item3:pro:diameter}  For any~$K_0$, there exists a constant $A=A(K_0)$ such that if $t=(a,b,c)$ is a
  positive triple with $\mathrm{K}(t)\leq K_0$, then for every $\tau_0$ and
  $\tau_1$ in  $D(t, K_0)$, we have, on $V_{b}(a,c)$ 
  \[g_{\tau_0}\leq A g_{\tau_1}\ .
  \]
  Furthermore $A(K_0)$ tends to~$1$ as $K_0$~goes to~$0$. 
\end{enumerate}
\end{proposition}

\begin{proof} Let $t=(a,b,c)$ be a positive triple. Let $\seq{\tau}$ be a sequence of tripods such that $(\tau^-_m,\tau^+_m)=(a,c)$ and 
$$
\seqm{d_{\tau_m}(b,\tau^0_m)}\  , $$ is bounded.
Let $\seq{g}$ be a sequence of elements in $\operatorname{Aut}_0(\mk g)$,
stabilizing~$a$ and~$c$ and such that $\seqm{g^{-1}_m(\tau^0_m)}$ is constant and let~$\tau^0$ be this constant. Let $\tau\defeq (a,\tau^0,c)$. It follows that 

$$
\seqm{d_{\tau}(g_m^{-1}(b),\tau^0)} \ ,
$$
is bounded. Since $d_\tau$ is a proper metric (\emph{i.e.}\ every bounded set is relatively compact),
the sequence $\seqm{g_m^{-1}(b)}$ --- after extracting a subsequence ---
converges to $e$ with $(a,e,c)$ positive. Since $\operatorname{Aut}_0(\mk g)$ acts properly on the
space of tripods, it follows that $\seq{g}$ is bounded. Thus after taking a
subsequence $\seq{\tau}$ converges to a tripod $\tau_\infty$,  with
$\tau^0_\infty$ in $V_t$.
This proves that, for all~$K_0$, the set $D(t,K_0)$ is compact.
Since $D(t,K_0)$ is non-empty for $K_0>\mathrm{K}(t)$, it follows that the
decreasing intersection
$$
D(t)=\bigcap_{K_0>\mathrm{ K}(t)} D(t, K_0)
$$
is not empty.

The second assertion is an immediate consequence of the first.

The third follows from the first as a  consequence of the
second part of Proposition~\ref{pro:dm-complete}. 
\end{proof}

\subsubsection{The diamond metric for triples}

The following definition is one of the goal of this section.

\begin{definition}\label{def:diammet}
  Let $t$ be a positive triple.   The {\em diamond metric} (for the triple~$t$) $g_t$ is the Riemannian metric
  on~$V_t$ defined as follows: for every~$x$ in~$V_t$, the unit ball
  of~$g_{t,x}$ is the John ellipsoid of the union of the unit balls
  of~$g_{\tau,x}$ for $\tau$ varying in~$D(t)$.

The associated distance is the {\em diamond metric} $d_t$.
\end{definition}

Explicitly, one has $g_{\tau,x}\leq g_{t,x}$ for every~$\tau$ in $D(t)$ and
$g_{t,x}\leq g$ for every Euclidean scalar product~$g$ on $T_x V_t$ such that
$g_{\tau,x}\leq g$ for every~$\tau$ in $D(t)$. Furthermore $g_{t,x}$ is the
unique minimizer among the Euclidean scalar products~$g$ satisfying the
previous condition. It also follows immediately from
point~(\ref{item3:pro:diameter}) of Proposition~\ref{pro:diameter} that
$g_t\leq A g_\tau$ for every~$\tau$ in $D(t)$ with $A=A(K(t))$.

When~$t$ is a tripod, this definition agrees with the one of the previous
paragraph thanks to the second item of Proposition~\ref{pro:diameter}.

As an immediate corollary of Proposition~\ref{pro:diameter} and Proposition~\ref{pro:dm-complete}, we have
\begin{coro}
  The diamond metric is complete. Moreover if a sequence of positive triples
  $\seq{t}$ converges to a tripod $\tau$, then $\seqm{g_{t_m}}$ converges to
  $g_\tau$ on every compact of the diamond $V_\tau$.
\end{coro}
The following
Propositions~\ref{pro:Acontract2} and~\ref{pro:contract2} are two contractions
properties of the diamond metrics that we shall use in the sequel.

\begin{proposition}[\sc Contraction]\label{pro:Acontract2}
Let $\seq{t}$ be a sequence of positive triples, with $t_m=(a_m,b_m,c_m)$. Assume that the sequence 
	$\seqm{\mathrm{ K}(t_m)}$ of tripod defects is  bounded.  Assume that $V_{t_{m+1}}\subset V_{t_m}$ and that
\begin{equation}
\bigcap_{m\in\mathbb N}V_{t_m}=\{z\}\ .\label{eq:hypT2}
\end{equation}
For any positive $R$, let $V_{t_m}(R)$ be the ball of radius $R$ and center $a_m$ with respect to $d_{t_m}$. Then on $V_{t_m}(R)$, we have 
$$
g_{t_0}\leq k_m \cdot g_{t_m}\ ,
$$
with  $\seq{k}$ converges to zero.
\end{proposition}

\begin{proof} By Definition~\ref{def:diammet} of the diamond metric for triples, and
Proposition~\ref{pro:diameter} it follows that we can find a constant $A$,
such that for all $m$, we can find a tripod $\tau_m$ with the same extremities
as~$t_m $ and with $$
d_{\tau_m}(\tau^0_m,b_m)\leq A\ , \ \frac{1}{A} g_{\tau_m}\leq g_{t_m}\leq A g_{\tau_m}\ .
$$
The result now follows from the corresponding proposition for tripods: Proposition~\ref{pro:Acontract}.
\end{proof}

\begin{proposition}[\sc Contraction in corners]\label{pro:contract2}
  Let $\seq{t}$ be a sequence of positive triples, where
  $t_m=(t_{m}^{-},t_{m}^{0},t_{m}^{+})$. Assume that
  \begin{enumerate}
  \item the sequence $\seqm{\mathrm{K}(t_m)}$ of tripod defects is  bounded;
  \item the sequences $\seqm{t_{m}^{-}}$ and  $\seqm{t_{m}^{+}}$ converge to
    transverse points $a$ and $c$ respectively;
  \item There exists $\seq{u}$  a sequence of elements of $\gp$, such that  $u_m$ belongs
    to  $V_{t_m}$, the sequence $\seqm{d_{t_m}(t_{m}^{0},u_m)}$ uniformly
    bounded, and $\lim_{m\to\infty} u_m=a$.   
  \end{enumerate}
  Then $\lim_{m\to\infty}t_{m}^{0}=a$.
\end{proposition}

\begin{proof} By the first hypothesis and Proposition~\ref{pro:diameter}, we can find a constant~$A$, a  sequence of tripods $\seq{\tau}$ with $\tau_m^\pm=t_m^\pm$ and such that  
$$
d_{\tau_m}\leq A\  d_{t_m}\ .
$$
In particular, we have that $\seqm{d_{\tau_m}(t^{0}_{m},\tau^{0}_{m})}$ and
$\seqm{d_{\tau_m}(u_m,\tau^{0}_{m})}$ are uniformly bounded. The result now
follows by applying twice
Proposition~\ref{pro:dm-contract}. Indeed, since
$\seqm{d_{\tau_m}(t^{0}_{m},\tau^{0}_{m})}$ is  uniformly bounded, it follows that
$t^{0}_{m} = \Psi_{\tau_m}(k_m)$ with $\seq{k}$ bounded. Hence by
Proposition~\ref{pro:dm-contract} (applied to any converging subsequence of $\seq{k}$), with $k'_m=h$, yields that 
$$
\lim_{m\to\infty}\tau^{0}_{m}=a\ .
$$
Applying again Proposition~\ref{pro:dm-contract} to $\seq{k'}$ with 
$\Psi_{\tau_m}(k'_m)=u_m$ yields that 
 $$
 \lim_{m\to\infty} u_m=a\ .
 $$
This concludes the proof.\end{proof}

\section{Positive representations are Anosov}\label{sec:posano}

In this section we introduce the notion of positive representations of a
surface group. We then show that any $\Theta$-positive representation is
$\Theta$-Anosov, establishing Theorem~\ref{thm_intro:Anosov} and
Corollary~\ref{cor_intro:open} from the introduction. As in the introduction, $S$~is a connected
oriented closed surface of genus at least~$2$.

\begin{definition}[\sc Positive representations]
Let~$\G$ be a semi-simple Lie group admitting a positive
structure relative to~$\Theta$. A
representation $\rho\colon \pi_1(S) \to \G$ is {\em $\Theta$-positive} if
there exist a non-empty $\grf$-invariant subset~$A$ of $\bgrf$ and a positive
$\rho$-equivariant map~$\xi$ from~$A$ to~$\gp$.

\end{definition}

The set~$A$ is necessarily dense since the action of $\grf$ on $\bgrf$ is minimal.
We will often say that a representation is positive if it is $\Theta$-positive.

\subsection{Anosov representations}
\label{sec:anos-repr}

Let us recall at this stage the definition of a $\Theta$-Anosov
representation from \cite{Labourie:2006}.  For simplicity we restrict ourselves to the
case 
of representations of~$\pi_1(S)$. Let us equip the surface~$S$  with an auxiliary
hyperbolic metric. Let ${\mathsf U}S$ be the unit tangent bundle of~$S$
equipped with its geodesic flow~$\phi_t$. Let us also freely  identify the
space of cyclically oriented triples of $\bgrf$ with the
unit tangent bundle $\mathsf U{\mathbf H}^2$ of the universal cover of~$S$.

Let~$\rho$ be a representation of
$\Gamma\defeq\pi_1(S)$ in~$\G$. Let ${\mathcal F}_\Theta$ the flat
$\gp$-bundle over ${\mathsf U}S$ associated with $\rho$, and $\Phi_t$ the  parallel transport on ${\mathcal F}_\Theta$ along~$\phi_t$.

The representation $\rho$ is {\em $\Theta$-Anosov} if there exists a  $\rho$-equivariant continuous map, called the {\em limit map}.
\begin{align*}
\xi\colon \partial_\infty\Gamma&\to \gp\ ,
\end{align*}
such that the corresponding section~$\Xi$ of 
${\mathcal F}_{\Theta}$ (which is constant along the leaves of the weakly unstable foliation) satisfies the following {\em contraction property}:
there exist an open neighborhood~$\mathcal V$ of the image of~$\Xi$ that is a
fiber bundle over ${\mathsf U}S$ with fiber $\mathcal{V}_x$ for~$x$ in
${\mathsf U}S$, a
continuous family of Riemannian metric $g_x$ on the fibers~$\mathcal{V}_x$, and some positive number
$T$ such that
$\Phi_{-T}(\mathcal V) \subset \mathcal V$, and, for all $x$ in $\mathcal V$, 
\begin{equation*}
  (\Phi_T)^* g_x\leq \frac{1}{2} g_{\phi_T(x)}\ ,
\end{equation*}
where~$g_x$ is $g$~restricted to~$\mathcal V_x$. 

\vskip 0.2 truecm

Note that there is another section $\Xi^*$, which is constant along the leaves
of the weakly stable foliation and is contracted under  $\Phi_{-T}$. 
	
Observe that, in general, the existence of a continuous equivariant map, even sending
distinct points to transverse points, is weaker than the condition of being
Anosov.

To establish the Anosov property for positive representations, we first extend the positive boundary map to
a left-continuous boundary map and to a right-continuous boundary map using
Corollary~\ref{coro:posmap}. We prove then that these extensions are
continuous (and thus coincide), and then deduce the Anosov property using the contraction property of the diamond metrics (Proposition~\ref{pro:Acontract2}).

\subsection{Properness}
The following definition will be used several times in the sequel: an application $f$ defined on a subset $A$ of a topological set $X$, with values in some topological set $Y$  is {\em bounded} if for every compact set $K$ in $X$, $f(A\cap K)$ is relatively compact.

\begin{lemma}\label{lem:UnifBd}
Let $A$ be a dense set in the circle. Let $\phi_\pm$ be  positive maps from $A$ to
$\gp$. We assume that, for all cyclically oriented quadruple $(x,y,z,t)$ in~$A$ and for any
choice of $\varepsilon$, $\eta$, $\nu$, and $\beta$ in $\{+,-\}$, the quadruple 
$$(\phi_\varepsilon(x),\phi_\eta(y),\phi_\nu(z),\phi_\beta(t))$$
is a positive. Let $A^3_+$ be  the set of  triples of  pairwise distinct elements of $A$.  

Then, for any
$\varepsilon$, $\eta$, and $\nu$ in $\{+,-\}$, $\phi_\varepsilon\times
\phi_\eta\times \phi_\nu$ is
bounded as a map from $A^3_+$ to the space~$\mathcal T$ of positive triples in~$\gp$.
\end{lemma}

\begin{proof} Let  $\chi=(x_1,x_2,y_1,y_2,z_1,z_2)$ be a cyclically oriented sextuplet in~$S^1$. Let $I_\chi$  be the subset of  $(S^1)^3$ given  by
\[
I_\chi=\{(X,Y,Z)\mid x_1<X<x_2<y_1<Y<y_2<z_1<Z<z_2\}\ .
\]
Observe that $I_\chi$ consists of cyclically oriented triples. 
Let 
\[
K = \phi_\epsilon \times \phi_\eta \times \phi_\nu \left(I_\chi\cap A_+^3\right).
\]
It is enough to show that $\overline K\subset\mathcal T$, where the closure is taken in $\gp^3$.

Let us fix, by density, $a_0$, $a_1$, $a_2$, $b_0$, $b_1$, $b_2$, $c_0$,
$c_1$, and $c_2$ in~$A$ such that 
\[
  (x_1,x_2, a_0,a_1,a_2, y_1,y_2, b_0,b_1,b_2, z_1,z_2,c_0,c_1,c_2)
\]
is cyclically oriented. To lighten notation, we set $\alpha_i=\phi_\epsilon(a_i)$,
$\beta_i=\phi_\eta(b_i)$, and $\gamma_i=\phi_\nu(c_i)$ for $i=0,1,2$.

From the positivity of the maps and thus of the image of the $15$-tuple
defined above, it follows that if $(x,y,z)$ belongs to $K$, then
$$
x\in V^*_{\alpha_1}(\gamma_2,\alpha_0)\ , \ y\in V^*_{\beta_1}(\alpha_2,\beta_0)\ , \ z\in V^*_{\gamma_1}(\beta_2,\gamma_0)\ . 
$$
Thus if $(a,b,c)$ belongs to $\overline K$, then 
\[
a\in \overline{V}{}^{*}_{\alpha_1}(\gamma_2,\alpha_0)\ , \ b\in \overline{V}{}^{*}_{\beta_1}(\alpha_2,\beta_0)\ , \ c\in \overline{V}{}^{*}_{\gamma_1}(\beta_2,\gamma_0)\ . 
\]
Using the inclusion Corollary~\ref{coro:inclus}, we get
\[
a\in V^*_{\alpha_1}(\gamma_1,\alpha_1)\ , \ b\in V^*_{\beta_1}(\alpha_1,\beta_1)\ , \ c\in V^*_{\gamma_1}(\beta_1,\gamma_1)\ . 
\]
By the necklace Corollary~\ref{coro:necklace}, $(a,b,c)$ is a positive
triple, \emph{i.e.}\ it belongs to $\mathcal T$.
This concludes the proof.
\end{proof}

\begin{proposition}\label{pro:triple-bounded}
	Let $\rho\colon \pi_1(S) \to \G$. Let $\xi$ be a positive $\grf$-invariant  map from $\bgrf$ to $\gp$. 
	
	Let $\mathcal T_{\grf}$ be the set of triples of  pairwise distinct
        elements in  $\bgrf$, and let~$\mathcal T$ be the set of positive triples in $\gp$. Let  $\Xi$ be the map from  $\mathcal T_{\grf}$ to $\mathcal T/\G$, defined by 
$$
	\Xi(x,y,z)\defeq [\xi(x),\xi(y),\xi(z)]\ .
	$$
Then the image of $\Xi$ is relatively compact.
\end{proposition}

\begin{proof}
The map $\Xi$ is invariant by the diagonal action of $\grf$. The result follows then from Lemma~\ref{lem:UnifBd} using the fact that $\grf$ acts cocompactly on $\mathcal T_{\grf}$. \end{proof}

\subsection{An \emph{a priori} bound on the tripod defect}

For any positive triple~$t$, let ${\rm K}(t)$ be the tripod defect 
introduced in paragraph~\ref{par:met-trip}. 
Then Proposition~\ref{pro:triple-bounded} implies an a priori bound on the tripod defect. 

\begin{proposition}\label{pro:a-priori}
	Let $\rho$ from $\pi_1(S)$ to $\G$. Let $\xi$ be a $\rho$-equivariant  positive map from a $\grf$-invariant dense subset~$A$ of $\bgrf$ to $\gp$. Then there exists a constant $K_0$ such that for all  triple~$t$ of pairwise distinct points in the closure of  $\xi(A)$, we have 
	$$
	{\rm K}(t)\leq K_0 \ .
	$$
\end{proposition}

\begin{proof}
  This is an immediate consequence of Proposition~\ref{pro:triple-bounded} and
  the fact that ${\rm K}$ is a continuous function on $\mathcal T$.
\end{proof}

\subsection{Continuity of equivariant positive maps}

Let $\rho$ be a $\Theta$-positive representation, $A$ a non-empty
$\pi_1(S)$-invariant subset of $\bgrf$ and $\xi\colon A \to \gp$ the positive
$\rho$-equivariant boundary map. Then, by Corollary~\ref{coro:posmap},  there
exist a unique right-continuous $\rho$-equivariant boundary map $\xi_+\colon
\bgrf \to \gp$ and a unique left-continuous $\rho$-equivariant boundary map
$\xi_-\colon \bgrf \to \gp$, coinciding with the map~$\xi$ on dense subset.

Let $\mathcal T_{\grf}$ be the set of  triples of pairwise distinct points of $\bgrf$. For $t=(x,y,z)$ in  $\mathcal T_{\grf}$, let us define 
$$
\tau(t)=(\xi_+(x),\xi_+(y),\xi_+(z))\ .
$$

\begin{lemma}\label{lem:fbounded}
	The $\grf$-invariant function $f$ defined by
$$
f(x,y,z)=d_{\tau(t)}(\xi_+(y),\xi_-(y))
$$
is bounded: there is a constant~$D$ such that, for all $(x,y,z)$ in
$\mathcal{T}_{\grf}$, $f(x,y,z)\leq D$.
\end{lemma}
\begin{proof}
Let $\mathcal{Q}$ be the set of quadruples $(a,b,c,d)$ in $\gp$ such
  that there exists a diamond $V$ with extremities~$a$ and~$d$ and containing
  both~$b$ and~$c$.   Using Lemma~\ref{lem:UnifBd}, we see that the map $$(x,y,z)\mapsto (\xi_+(x),
  \xi_+(y),\xi_-(y), \xi_+(z))\ ,$$ from $\mathcal{T}_{\grf}$ to 
  $\mathcal{Q}$  is bounded. As the real valued function on $\mathcal{Q}$  sending a quadruple $(a,b,c,d)$ to 
   $d_{(a,b,d)}(b,c)$ is continuous, we get the result.\end{proof}

\begin{lemma}
  	The map $\xi_+$ is continuous.
\end{lemma}
\begin{proof}
  Since $\xi_+$ is right-continuous we only have to prove that it is left-continuous.
 Let $x$ and~$y$ be in $\bgrf$, and let $\seq{x}$ be a sequence in 
 $\bgrf$,  such that  $(x_m,x,y)$ is cyclically oriented with respect to the orientation on $\bgrf$, and that  $\seq{x}$ converges to $x$. Let $t_m=(\xi_+(x_m),\xi_+(x),\xi_+(y))$. 

Recall that by Corollary~\ref{coro:posmap}, $\seqm{\xi_+(x_m)}$ converges to $\xi_-(x)$.  

We now apply  Proposition~\ref{pro:contract2}  to the following setting:
\begin{equation*}
t_{m}^{-}=\xi_+(x_m)\ , \ t_{m}^{0}=\xi_+(x)	\ , \ u_m=\xi_-(x)\ , \ t_{m}^{+}=\xi_+(y)\ .
\end{equation*}
Since  $$
\seqm{d_{t_m}(t_{m}^{0},u_m)}=
\seqm{d_{t_m}(\xi_+(x),\xi_-(x))}$$ is bounded by Lemma~\ref{lem:fbounded} and $\seqm{{\rm K}(t_m)}$ is bounded by Proposition~\ref{pro:a-priori}, we get that 
$$
\lim_{m\to\infty}\xi_+(x_m)=\xi_+(x)\ .
$$
This proves that $\xi_+$ is left-continuous.
 \end{proof}
  
  As a consequence we obtain 

\begin{proposition}
\label{pro:continuous}
Let $\rho$ from $\grf$ to $\G$ be a positive representation and $\xi$ the positive $\rho$-invariant boundary map from a $\grf$-invariant dense subset of $\bgrf$ to $\gp$. Then $\xi$ extends to a $\rho$-equivariant positive continuous map from $\bgrf$ to $\gp$.
\end{proposition}
The extended map~$\xi$ from $\bgrf$ to~$\gp$ will be called
{\em the positive boundary map} of~$\rho$.

\subsection{The Anosov property}
We are now in position to prove Theorem~\ref{thm_intro:Anosov} from the introduction. More precisely we show 
\begin{proposition}\label{prop:Anosov}
 
Let $\rho$ from $\grf$ to  $\G$ be a positive representation and $\xi$ from $\bgrf$ to $\gp$ be the $\rho$-equivariant continuous positive boundary map. Then $\rho$ is $\Theta$-Anosov and its boundary map is $\xi$.
\end{proposition}

Let us start with a general lemma
\begin{lemma}\label{lem:little}
  Let $\seq{b^0}$ and $\seq{b^1}$ be two sequences in~$\gp$
  converging to~$c$. Let~$d_0$ and~$d_1$ be in~$\gp$ such that
  $(d_0,c, d_0)$ is a positive triple and assume that, for all~$m$ in~$\mathbb{N}$,
  $$
  (d_0,b^0_m,b^1_m, d_1)
  $$
  is a positive quadruple.
Let $V_m$ be the unique diamond with extremities~$b^0_m$ and~$b^1_m$ contained in
the diamond $V_{c}(d_0,d_1)$.  Then
	$$
	\lim_{m\to\infty} V_m=\{c\}\ .
	$$
\end{lemma}
\begin{proof}
  Let $a_0$ be in~$V_{c}^*(d_0, d_1)$ and $a_1$ be in
  $V_{c}^{*}(a_0, d_1)$ so that $(a_0, d_0, c, d_1, a_1)$ is a positive
  quintuple and, for all big enough~$m$, $(a_0, d_0,b^0_m,b^1_m, d_1, a_1)$ is
  a positive configuration.

  Let $z_m$   belong to~$V_m$, we want to prove that 
$$
\lim_{m\to\infty} z_m=c\ .
$$

Let $p$  in $\overline{V}_{c}(d_0,d_1)$ be an accumulation point of the sequence~$\seq{z}$. Up to extracting a subsequence 
we may assume  
$$
\lim_{m\to\infty} z_m=p\ ,
$$	
By Corollary \ref{coro:inclus}, $p$~belongs to $V_c(a_0, a_1)$ and
in particular it belongs to $\Omega_{a_0}\cap\Omega_{a_1}$.
From the fact that $z_m$~belongs to~$V_{d_0}(a_0, b_{m}^{1})$ we get that
$p$~belongs to $\overline{V}_{d_0}(a_0,c)$; similarly $p$~belongs to
$\overline{V}_{d_1}(a_1,c)$. Therefore
$$
p\in \overline{V}_{d_0}(a_0,c) \cap \overline{V}_{d_1}(a_1,c)\cap  \Omega_{a_0}\cap\Omega_{a_1}\ .
$$
Let $V=V_{d_0}(a_0,c)$, and recall that by Lemma~\ref{lem:semigroup}, 
$$V_{d_1}(a_1,c)\subset V^*\ .$$
Finally remark that 
$$
\overline{V}\cap \Omega_{a_1}=\overline{\ms N}_{a_1} \cdot c\ , \overline{V}{}^{*}\cap \Omega_{a_1}=\overline{\ms N}^{-1}_{a_1} \cdot c\ ,
$$
for the (positive) semigroup $\ms N_{a_1}$ in $\ms U_{a_1}$.
Since 
$$
\overline{\ms N}_{a_1}\cap \overline{\ms N}^{-1}_{a_1}=\{\id\}\ ,
$$
one has $p=c$, which is what we wanted to prove. \end{proof}

\begin{proof}[Proof of Proposition \ref{prop:Anosov}]
  The chosen hyperbolization of~$S$ defines a $\grf$-invariant
  cross-ratio on $\bgrf\cong \rp$. Let us also fix   an orientation on $\bgrf$. For any cyclically oriented
  triple $t=(x,y,z)$, let us consider the harmonic (with respect to the
  cross-ratio)  quadruple $(x,y,z,w)$, and let then
  $$Y_t \defeq V_{\xi(z)}(\xi(y),\xi(w))\ .$$
  By construction $Y_t$ is an open neighborhood of $\xi(z)$. Moreover if
  $(x,y_1,y_0,z)$ is a cyclically oriented quadruple,
  \begin{equation}
    Y_{(x,y_0,z)}\subset Y_{(x,y_1,z)}\ .\label{eq:YtYs}
  \end{equation}
  Finally, since $\xi$ is continuous, by Lemma~\ref{lem:little} 
  \begin{equation}
    \lim_{y\to z}Y_{(x,y,z)}=\{\xi(z)\}\ .\label{eq:Ytxx}
  \end{equation}
  	
  We now deduce the Anosov property from Assertion~\eqref{eq:Ytxx}.

  Recall that the chosen  uniformization of the surface enables
 us to identify the space of triples in the boundary at infinity with the
unit tangent bundle $\mathsf U{\mathbf H}^2$ of the universal cover of~$S$.  Let
$\{\phi_s\}_{s\in\mathbb R}$ be the geodesic flow on $\mathsf U{\mathbf
  H}^2$. Let $\mathcal F$ be the trivial bundle $\gp\times \mathsf U{\mathbf
  H}^2$. The actions of $\grf$ on $\mathsf U{\mathbf H}^2$ and on~$\gp$ ---through~$\rho$--- give rise to an action of $\grf$ on~$\mathcal F$.

Let~$\mathcal U$ be the subbundle of~$\mathcal{F}$ with open fibers given
by
$$
\mathcal U=\{(x,v)\in \mathcal F\mid v\in \mathsf U{\mathbf H}^2,\ \ x\in Y_v\}\ .
$$
The bundle $\mathcal U$ is invariant by the $\grf$-action, moreover it has a canonical section $\sigma_0$ given by
$$
\sigma_0(x,y,z)=\xi(z)\ .
$$
Let us lift the flow $\{\phi_s\}_{s\in\mathbb R}$ to a flow $\{\Phi_s\}_{s\in\mathbb R}$ on $\mathcal F$ acting trivially on the first factor. By assertion~\eqref{eq:YtYs}, for all positive $s$
$$
\Phi_{-s}(\mathcal U)\subset\mathcal U\ .
$$
Moreover the section $\sigma_0$ is invariant by $\{\Phi_s\}_{s\in\mathbb R}$.

The diamond metric  $g_t$ and the diamond distance $d_t$ on each $Y_t$ give a metric on each fiber of $\mathcal U$ which depends continuously on the base and is equivariant under the  action of $\grf$.

For any $R$, let $\mathcal U(R)$ be the neighborhood of the image of the section~$\sigma_0$, given by
$$
\mathcal U(R)=\{(x,v)\in \mathcal U\mid v\in \mathsf U{\mathbf H}^2,\ \ d_v(x,\sigma_0(v))\leq R\}\ .
$$
It now follows from assertion~\eqref{eq:Ytxx} and Proposition~\ref{pro:Acontract2},  that for every~$u$ in $\mathsf U{\mathbf H}^2$, there is $s_u$ such that, for all $(x,u)$ in $\mathcal U(R)$
\begin{equation}
	\hbox{for all $s\geq s_u$,}\ \ g_{\Phi_s(x,u)}\circ {\mathsf T}_{(x,u)} \Phi_{-s}\leq \frac{1}{2}	g_{(x,u)}\ \ \ .\label{eq:su}
\end{equation}
Let now $s$ be the real valued function on  $\mathsf U{\mathbf H}^2$ defined by 
$$
s(u)=\inf\{s_u\mid s_u \hbox{ satisfies assertion~\eqref{eq:su}} \hbox{ on } \mathcal U(R)\}\ . $$
The function $u\mapsto s(u)$ is upper  semicontinuous  and  invariant under the action of $\pi_1(S)$. Thus by compactness of $\mathsf U{\mathbf H}^2/\pi_1(S)$  the function has an upper bound $s_0$. Then for all $s$ greater than $s_0$
\begin{equation*}
	\Phi_s^*g \leq 	\frac{1}{2}	g \ ,
\end{equation*}
on $\mathcal U(R)$. In other words, the action of $\{\Phi_{-s}\}_{s\in\mathbb
  R}$ is contracting on $\mathcal{V} =\mathcal{U}(R)$ and $\sigma_0$ is an invariant section.

Thus $\rho$ is $\Theta$-Anosov according to the definition given in the beginning of the section and $ \xi$~is its limit curve. 
 \end{proof}

Now Corollary~\ref{cor_intro:open} in the introduction follows directly from
the openness of the set of $\Theta$-Anosov representations. More precisely
  \begin{proof}[Proof of Corollary~\ref{cor_intro:open}] A positive
  representation $\rho_0$ with limit map $\xi_0$, is $\Theta$-Anosov. Thus there is an
  open neighborhood  $U$ of $\rho_0$ containing only $\Theta$-Anosov
  representations. For any $\rho$ in this neighborhood, let $\xi_\rho$ be the
  limit map. Note that this map is equivariant and transverse. The map
  $\rho\mapsto\xi_\rho$ which sends an Anosov representation to its limit
  curve is continuous. Since $\xi_{\rho_0}$ is positive it sends pairwise distinct triples in
  $\bgrf$ to positive triples. Moreover $\pi_1(S)$ acts cocompactly on the set
  of triples of pairwise distinct points of $\bgrf$, thus, for~$\rho$ close
  enough to $\rho_0$, the continuous curve $\xi_\rho$ sends pairwise distinct triples to positive triples. 
  Hence by Proposition~\ref{pro:triplenough},  $\xi_\rho$ is positive. 
  \end{proof}

\begin{remark}
The definition of $\Theta$-positive representations can be made in more
generality for non-elementary word hyperbolic group $\Gamma$ whose boundary
admits a cyclic ordering. This holds if $\Gamma$ is a surface group, but also
if $\Gamma$ is virtually free.
For example, an appropriate extension of the arguments in this section shows
that a representation of $\Gamma$ is $\Theta$-Anosov if it admits a $\rho$-equivariant positive boundary map $\xi\colon \partial_\infty \Gamma \to \gp$.
\end{remark}

\section{Closedness}\label{sec:closedness}

In this section we consider the space $\operatorname{Hom}^\Theta (\pi_1(S),\G)$ of
homomorphisms from $\pi_1(S)$ to $\G$,  which contains  a $\Theta$-loxodromic
element.

We show that the set of $\Theta$-positive representations  $\operatorname{Hom}_{\Theta\textrm{-pos}}(\pi_1(S),\G)$ is an open and closed subset of $\operatorname{Hom}^\Theta(\pi_1(S),\G)$, hence a union of connected components. 

We first have 

\begin{proposition}\label{prop:nonpara}
Every $\Theta$-positive representation  is an element of the set   $\operatorname{Hom}^\Theta(\pi_1(S),\G)$. Moreover, every $\Theta$-positive representation  has a compact centralizer and  does not  factor through a proper parabolic subgroup of $\G$.
\end{proposition}
 
\begin{proof} 
 
The   first part is a consequence of Proposition \ref{pro:prox-pos}.
Let us first note that since the centralizer of a positive triple is compact, the centralizer of a positive representation is compact as well. 
Let $\rho\colon \pi_1(S) \to \G$ be a positive representation with
$\rho$-equivariant positive boundary map $\xi\colon \bgrf \to \gp$. Then
$\rho$ is $\Theta$-Anosov with boundary map $\xi$. This remains true when
restricting the representation to a finite index subgroup.  To argue by
contradiction we can thus assume that without loss of generality
$\rho(\pi_1(S))$ is contained in a proper parabolic subgroup of $\G$. We
consider the semi-simplification $\rho^{ss}$ of~$\rho$, whose image is
contained in a Levi factor of the parabolic subgroup. 

By \cite[Proposition 1.8]{ggkw_anosov} the semi-simplification~$\rho^{ss}$ is
$\Theta$-Anosov, denote $\xi^{ss}$ the $\rho^{ss}$-equivariant boundary map.
Since $\rho^{ss}$ belongs to 
 the closure of the $\G$-orbit of $\rho$;  there exists
 thus a sequence $\seq{g}$ in~$\G$ such that $\rho^{ss}$ is the limit of $g_m \rho g_{m}^{-1}$. Since the boundary map $\xi^{ss}$ is transverse, Lemma~\ref{lem:open-trip} implies that the 
 boundary map $\xi^{ss}$ is also positive  as well. But this is a contradiction because the centralizer of $\rho^{ss}$ in $\G$ contains the center of the Levi factor of the parabolic subgroup which is non-compact.
\end{proof} 

By a classical result of Borel and Tits \cite[Corollaire 3.3]{Borel:1971vd}
(proved also by Morozov~\cite{MorozovRegularity} in characteristic zero), the set
$\operatorname{Hom}^*(\pi_1(S),\G)$ is contained in the set 
of reductive homomorphisms, \emph{i.e.}\ representations $\rho\colon \pi_1(S) \to \G$, whose Zariski closure is reductive. Thus a direct consequence of Theorem~\ref{prop:nonpara} is 

\begin{coro}
Let $\rho\colon \pi_1(S) \to \G$ be a $\Theta$-positive representation, then the Zariski closure of $\rho(\pi_1(S))$ is reductive.
\end{coro}

  We also show the following result as a consequence of a result of Benoist--Labourie  \cite{Benoist:1993}.

\begin{proposition}\label{pro:BL}
Assume that the Zariski closure $\mathsf H$ of the image of $\rho$ is such that the exponential of every element in the open Weyl chamber of $\mathsf H$ is loxodromic with respect to $\gp$. Then $\rho$ belongs to  $\operatorname{Hom}^\Theta(\pi_1(S),\G)$.

 In particular, every representation with Zariski dense image belongs to $\operatorname{Hom}^\Theta(\pi_1(S),\G)$.	
\end{proposition}

\begin{proof}
	Indeed by \cite[Theorem A.1.1]{Benoist:1993}, an element $h$ of the image of $\rho$ has an hyperbolic part which belongs to the Weyl Chamber. Hence $h$ is loxodromic.  
\end{proof}
 
We expect that the list of possible Zariski closures of $\Theta$-positive representations is indeed restrictive. Classifications of the Zariski closures for maximal representations were given in \cite{Burger_tight,Burger:2010ty, Hamlet_1, Hamlet_2} and for Hitchin representations in \cite{Guichard_Zariski, sambarino2020infinitesimal}.

Since the set of $\Theta$-positive representations  is open in  $\operatorname{Hom}(\pi_1(S),\G)$ by Corollary~\ref{cor_intro:open}, it is also open in $\operatorname{Hom}^*(\pi_1(S),\G)$.

We will now show 
 
\begin{theorem}{\label{theo:conn}}
The set of $\Theta$-positive homomorphisms is closed in the set $\operatorname{Hom}^\Theta(\pi_1(S),\G)$.
\end{theorem}
 
We will first prove the following proposition of independent interest:

\begin{proposition}{\label{pro:vonv-transv}}
 	 Let $\seq{\rho}$ be a sequence of $\Theta$-positive representations converging to a representation $\rho_\infty$. Let $\xi_m$ be the limit curve of $\rho_m$. Assume that we can find  $x_0$ and $y_0$ in $\bgrf$ such that $\left\{(\xi_m(x_0),\xi_m(y_0))\right\}_{m\in\mathbb N}$ converges to a transverse pair, then $\rho_\infty$ is positive.
 \end{proposition}

\subsection{Proof of proposition \ref{pro:vonv-transv}}

We fix a countable set $A$ in $\bgrf$, invariant by $\grf$ and containing $x_0$ and $y_0$. We may now assume, by the Cantor diagonal argument, that $\seqm{\xi_m|_A}$ 
converges simply to a map $\xi_\infty$ from $A$ to $\gp$. By hypothesis $\xi_\infty(x_0)$ and $\xi_\infty(y_0)$ are transverse.

For any pair of distinct points $(x,y)$ in $A^2$, denote by $]x,y[$ the
interval in the oriented circle $\bgrf$ with origin~$x$ and extremity~$y$, let $c$ be in $A \cap ]x,y[$ and
set 
$$
W_\infty(x,y){\defeq}\lim_{n\rightarrow\infty} \overline{V}{}^{*}_{\xi_n(c)}(\xi_n(x),\xi_n(y))\ ,
$$
the convergence being for the Hausdorff topology, and using again the Cantor
diagonal extraction, we can and will assume that all those sequences converge of all
$(x,y)$ in~$A^2$ with $x\neq y$.
Observe that $W_\infty(x,y)$ only depends on $x$, $y$, and the interval $]x,y[$. 
Furthermore the following
equivariance property holds: $\rho_\infty(\gamma) W_\infty( x,y) = W_\infty(
\gamma\cdot x, \gamma.\cdot y)$.

\begin{lemma}\label{lem:ZD3}
	Assume that $\xi_\infty(x)$ and $\xi_\infty(y)$ are transverse then $W_\infty(x,y)$ is a closure of a diamond with extremities $\xi_\infty(x)$ and $\xi_\infty(y)$ and is Zariski dense.
\end{lemma}

\begin{proof} Since $\xi_\infty(x)$ and $\xi_\infty(y)$ are transverse,  $W_\infty(x,y)$ is the closure of a diamond (see Proposition~\ref{pro:bounded}). It thus contains an open set, and in particular is Zariski dense. 
\end{proof}

\begin{lemma}{\label{lem:ZarDens}}
For every pair of distinct points $(x,y)$ and $(z,t)$ in $A$, one has
$$
\overline{W}{}_{\infty}^{Z}(x,y)=\overline{W}{}_{\infty}^{Z}(z,t)
$$
where  $\overline{M}{}^Z$ denotes the  Zariski closure of a set $M$.
In particular, for all distinct $x$ and $y$,  $W_\infty(x,y)$ is Zariski dense.
\end{lemma}
Observe that only the last assertion depends on the assumption that $\xi_\infty(x_0)$ and $\xi_\infty(y_0)$ are transverse.

\begin{proof} We shall use freely the following fact. If $\gamma$ is an algebraic automorphism of a variety $V$, if $B$ is a Zariski closed subset such that 
$\gamma(B)\subset B$ then $\gamma(B)=B$.

We first prove that if  $[u,v]\subset [w,s]$, then  we have
\begin{equation}
	\overline{W}{}_{\infty}^{Z}(u,v)=\overline{W}{}_{\infty}^{Z}(w,s)\ .\label{eq:Zincl1}
\end{equation}

We can always find an element $\gamma$ of $\grf$ such that
$$
\gamma[w,s]\subset [u,v]\ .
$$
Thus
$$
\rho_\infty(\gamma)\bigl(\overline{W}{}_{\infty}^{Z}(w,s)\bigr)\subset \overline{W}{}_{\infty}^{Z}(u,v)\subset\overline{W}{}_{\infty}^{Z}(w,s)\ .
$$
By the initial observation we get that 
$$
\overline{W}{}_{\infty}^{Z}(w,s) \subset \overline{W}{}_{\infty}^{Z}(u,v)  \subset\overline{W}{}_{\infty}^{Z}(w,s)\ ,
$$
and thus the assertion~\eqref{eq:Zincl1} follows. Take now $\gamma$ in $\grf$ such that 
$$
\gamma[x,y]\subset [x,y]\ , \ \ \gamma[x,y]\cup [z,t] \not=\bgrf\ .
$$
We can then find distinct points $u$ and $v$ such that $$
\left(\gamma[x,y]\cup [z,t]\right)\subset [u,v]\ .
$$
Thus, applying thrice assertion~\eqref{eq:Zincl1}, we have
$$
\overline{W}{}_{\infty}^{Z}(x,y) = \overline{W}{}_{\infty}^{Z}(\gamma\cdot
x,\gamma\cdot y) = \overline{W}{}_{\infty}^{Z}(u,v)=\overline{W}{}_{\infty}^{Z}(z,t)\ .
$$
The last assertion follows from the fact that $\xi_\infty(x_0)$ and $\xi_\infty(y_0)$ are transverse and thus $W_\infty(x_0,y_0)$ is Zariski dense by Lemma \ref{lem:ZD3}.
 \end{proof}

We are now in the position to show that $\rho_\infty$ is
$\Theta$-positive. This will be a consequence of the following proposition: 
\begin{proposition}\label{pro:ZD2}
  For any pair of distinct  points $(x,y)$, the
  pair $(\xi_\infty(x),\xi_\infty(y))$ is  transverse. Moreover, $\xi_\infty$ is a positive map.
 \end{proposition}
\begin{proof}
Let $(x, y, z)$ be a triple of pairwise distinct points in $\bgrf$. Let us denote for simplicity $x_n=\xi_n(x)$, $y_n=\xi_n(y)$ and $z_n=\xi_n(z)$ for $n$ in $\mathbb N\cup\{\infty\}$.
We choose diamonds by letting
\begin{equation*}
V^0_n=V^*_{z_n}(x_n,y_n)\ ,\ \ V^1_n=V^*_{y_n}(x_n,z_n)\ ,\ \
V^2_n=V^*_{y_n}(z_n,y_n)\ .
\end{equation*}
Since $W_\infty(x,y)$, $W_\infty(y,z)$, and $W_\infty(z,x)$ are Zariski dense, we can pick three points $a$, $b$, and $c$ so that 
\begin{enumerate}
\item $a\in W_\infty(x,y)$, $b\in W_\infty(y,z)$, $c\in W_\infty(z,x)$,
\item $a,b,c$ are pairwise transverse,
\item any point in $\{a,b,c\}$ is transverse to any point in $\{x_\infty,y_\infty,z_\infty\}$.
\end{enumerate}
Let now pick sequences $\seq{a}$, $\seq{b}$, and $\seq{c}$, with $a_m\in
V^0_m$, $b_m\in V^1_m$, and $c_m\in V^2_m$, and converging to~$a$, $b$, and~$c$ respectively.

We will now apply the necklace property several times. By
Proposition~\ref{coro:necklace}, $(a_m, b_m ,c_m)$ is a positive triple and since $a,b,c$ are pairwise transverse it follows that $(a,b,c)$ is a positive triple. 

Then, since $x_m$ belongs to $V^*_{b_m}(a_m,c_m)$, it follows that $x_\infty$ belongs to $\overline{V}{}^{*}_b(a,c)$. Since $x_\infty$ is transverse to both $a$ and $c$, $x_\infty$ belongs to $V^*_b(a,c)$. Symmetrically $y_\infty$ belongs to $V^*_c(a,b)$, $z_\infty$ belongs to $V^*_a(c,b)$. Applying Proposition~\ref{coro:necklace} again, $(x_\infty,y_\infty,z_\infty)$ is a positive triple.

The fact that, for any cyclically oriented quadruple $(x,y,z,w)$,  the
quadruple $(\xi_\infty(x), \xi_\infty(y), \xi_\infty(z), \xi_\infty(w))$ is positive, now follows from Proposition~\ref{pro:basic}.(\ref{it:basic2}). Hence the positivity of $\xi_\infty$ by definition.\end{proof}

\subsection{Proof of Theorem \ref{theo:conn}}

We consider a sequence $\seq{\rho}$  of $\Theta$-positive representations converging  to a representation $\rho_\infty$. Let $\seq{\xi}$ be the corresponding sequence of positive limit maps.
  Our assumption is  that image of $\rho_\infty$ contains a $\Theta$-loxodromic element $\rho_\infty(\gamma_0)$.

We fix a countable set $A$ in $\bgrf$, invariant by $\grf$, and containing $\gamma_0^+$ and $\gamma_0^-$. We may now assume applying the Cantor diagonal argument,  that $\seqm{\xi|_A}$ 
converges simply 
to a map $\xi_\infty$ from $A$ to $\gp$.

Observe now that if $y^+$ is the attracting fixed point of $\rho_\infty(\gamma_0)=\lim_{n\to\infty} \rho_\infty(\gamma_0)$, then $y^+$ is the limit of the attracting fixed points of $\{\rho_m(\gamma_0)\}_{m\in\mathbb N}$. It follows that $y^+=\xi_\infty(\gamma^+)$. The same holds for the repelling fixed point $y^-$ of $\rho_\infty(\gamma_0)$, and we have $y^-=\xi_\infty(\gamma^-)$. Since $y^+$ and $y^-$ are transverse, we can apply Proposition  \ref{pro:vonv-transv} to obtain that $\rho_\infty$ is positive.

\subsection{Proof of Theorem \ref{thm_intro:nonpara}}
 
By Theorem \ref{theo:conn} the set of $\Theta$-positive representations is closed in $\operatorname{Hom}^\Theta(\grf,\G)$. Using furthermore Corollary \ref{cor_intro:open} this set is  open. Thus the set of $\Theta$-positive representations is a union of connected components of  $\operatorname{Hom}^\Theta(\grf,\G)$. Since finally, we can obtain positive representation by factoring in a positive $\psld$ we deduce the Theorem.

\section{Positive representations and Cayley components}\label{sec:connected}
Let us recall that for a real split Lie group $\G$, the Hitchin component was originally defined by Hitchin as the
image of the Hitchin section~$\Phi$ which assigns to a tuple of holomorphic
differentials on a Riemann surface $\Sigma$ a $\G$-Higgs bundle on
$\Sigma$. Let us denote the image of $\Phi$ by $\mathcal{P}(\Sigma,
\G)$. Through the non-Abelian Hodge correspondence the set
$\mathcal{P}(\Sigma, \G)$ corresponds to a subset of the representation
variety $\operatorname{Rep}^+(\pi_1(S),\G)$, which we denote by the same
symbol. Hitchin showed that $\mathcal{P}(\Sigma, \G)$ is open and closed
(hence a union of connected components) and the map $\Phi$ gives a  parametrization of   $\mathcal{P}(\Sigma, \G)$. 
In the case of maximal representations a similar but more complicated
parametrization of the space of maximal representations was obtained in
\cite{BGPG2006}, \cite{GarciaPrada:2013be}, and \cite{Biquard:2017vs}. For any simple Lie groups admitting a positive
structure relative to~$\Theta$, the authors of 
 \cite{BCGGO_general} define in a similar way
subsets $\mathcal{P}_e(\Sigma,\G)$ of the moduli space of $\G$-Higgs bundles
by giving explicit parametrizations, see also \cite{Collier} and  \cite{ABCGGO} for indefinite orthogonal groups. They prove that 
$\mathcal{P}_e(\Sigma,\G)$ is open and closed in
$\operatorname{Rep}^+(\pi_1(S),\G)$. They further prove that all
representations in $\mathcal{P}_e(\Sigma,\G)$ have compact centralizer and
thus do not factor through a proper parabolic subgroup. They further show that the set $\mathcal{P}_e(\Sigma,\G)$ contains an open subset of $\Theta$-positive representations.

Theorem~\ref{thm_intro:nonpara} implies that any connected components of $\mathcal{P}^\Theta_e(\Sigma,\G) = \mathcal{P}_e(\Sigma,\G)\cap \operatorname{Rep}^\Theta(\pi_1(S),\G)/\G$ that contains a $\Theta$-positive representation consists entirely of $\Theta$-positive representations. Due to the extension of our main result in ~\cite{BGLPW_pos} we further have that any connected components of $\mathcal{P}_e(\Sigma,\G)$ that contains at least one $\Theta$-positive representation consists entirely of $\Theta$-positive representations. For many $G$ this implies that $\mathcal{P}_e(\Sigma,\G) \subset \operatorname{Hom}_{\Theta\textrm{-pos}}(\Gamma,\G)$. 

For this let us introduce the standard components of $\mathcal{P}_e(\Sigma,\G)$. 
Consider an
embedding of $\sld$, such that the induced map from~$\rp$ to~$\gp$ is a
positive circle, then the corresponding Fuchsian representation is
positive. These Fuchsian representations can now in addition be twisted by a
representation of~$\pi_1(S)$ into the centralizer of this $\sld$ in $\G$. This is
called a twisted positive Fuchsian representation. We call a component of  $\mathcal{P}_e(\Sigma,\G)$ standard if it contains a twisted positive Fuchsian representation. 

When $\G$ is a classical group and not locally isomorphic to $\ms{Sp}_4(\mathbb{R})$, or $ \ms
{SO}(p, p+1)$, every component of $\mathcal{P}_e(\Sigma,\G)$ is
standard~\cite{BCGGO_general}. 

For $\ms{Sp}_4(\mathbb{R})$  with the positive
structure relative to~$\Theta \neq \Delta$, positive representations correspond precisely to maximal representations \cite{Burger:2005, Burger:2010ty}. In particular, the exceptional connected components of maximal representations in $\mathcal{P}_e(\Sigma,\G)$
which do not contain any twisted positive Fuchsian
representation~\cite{Gothen:2001}, are positive. Similarly for the exceptional Hermitian Lie group of tube type, $\mathcal{P}_e(\Sigma,\G) = \operatorname{Hom}_{\Theta\textrm{-pos}}(\Gamma,\G)$ is the set of maximal representations. 

For $ \ms {SO}(p, p+1)$ with the positive
structure relative to~$\Theta \neq \Delta$, there also exist exceptional connected components in $\mathcal{P}_e(\Sigma,\G)$
 which do not contain any twisted positive Fuchsian representation. To deduce that they are positive we can use an embedding argument. Embedding 
$ \ms {SO}(p, p+1) \to \ms {SO}(p, p+2) $, these components are sent to standard components for  $\ms {SO}(p, p+2)$ ~\cite{BCGGO_general}, and thus they consist also entirely of $\Theta$-positive representations seen in $\ms {SO}(p, p+2) $.  Since any $\Theta$-positive representation in $\ms {SO}(p, p+2) $, whose image is contained in $ \ms {SO}(p, p+1)$ is also $\Theta$-positive as a representation into $ \ms {SO}(p, p+1)$, we conclude that these exceptional components consist entirely of $\Theta$-positive representations. 

However, for exceptional groups $G$ whose restricted root system have a Dynkin diagram of type $F_4$, we do not know if all connected components of $\mathcal{P}_e(\Sigma,\G)$ contain positive representations. 
If this where the case, this would imply $\mathcal{P}_e(\Sigma,\G) \subset \operatorname{Hom}_{\Theta\textrm{-pos}}(\Gamma,\G)$. We further expect to have $\mathcal{P}_e(\Sigma,\G) = \operatorname{Hom}_{\Theta\textrm{-pos}}(\Gamma,\G)$. For this to hold one would need to show that any positive representation lies in $\mathcal{P}_e(\Sigma,\G)$, or what might be the better approach, that none of the connected components of $ \operatorname{Hom}^+(\Gamma,\G)\backslash \mathcal{P}_e(\Sigma,\G)$ contain a $\Theta$-positive representation.

\bibliographystyle{amsplain}
\providecommand{\bysame}{\leavevmode\hbox to3em{\hrulefill}\thinspace}
\providecommand{\MR}{\relax\ifhmode\unskip\space\fi MR }
\providecommand{\MRhref}[2]{%
  \href{http://www.ams.org/mathscinet-getitem?mr=#1}{#2}
}
\providecommand{\href}[2]{#2}

\end{document}